\pdfoutput=1
\documentclass[final,12pt]{elsarticle}
\usepackage{lineno,hyperref,mathtools,amssymb,eqnarray,amsthm,enumerate,arydshln,chngpage,tikz,tikz-cd,fancyhdr}
\newtheorem{theorem}{Theorem}[section]
\newtheorem{lemma}[theorem]{Lemma} 
\newtheorem{proposition}[theorem]{Proposition} 
\newtheorem{corollary}[theorem]{Corollary} 
\theoremstyle{definition}

\theoremstyle{remark}
\newtheorem{remark}[theorem]{Remark}

\DeclareMathOperator{\Ima}{Im}

\DeclareMathOperator{\Ker}{Ker}
\journal{Journal of Algebra}

\begin{document}
	
	\begin{frontmatter}
		
		\title{Skew Braces and Hopf-Galois Structures\\ of Heisenberg Type}
		
		\author{Kayvan Nejabati Zenouz\footnote{Email: knejabati-zenouz@brookes.ac.uk}}
		
		\address{School of Engineering, Computing and Mathematics, Oxford Brookes University, Oxford, OX33 1HX}
		
		\begin{abstract}
			We classify all skew braces of Heisenberg type for a prime number $ p>3 $. Furthermore, we determine the automorphism group of each one of these skew braces (as well as their socle and annihilator). Hence, by utilising a link between skew braces and Hopf-Galois theory, we can determine all Hopf-Galois structures of Heisenberg type on Galois field extensions of fields of degree $ p^{3} $.
		\end{abstract}	
	\end{frontmatter}
	\tableofcontents{}
	\section{Introduction}\label{S1}
	Braces were introduced by W. Rump \cite{MR2278047}, as a generalisation of radical rings, in order to study the non-degenerate involutive set-theoretic solutions of the quantum Yang-Baxter equation. He also obtained a correspondence between these solutions and braces. Later, through the efforts of D. Bachiller, F. Ced\'o, E. Jespers, and J. Okni\'nski \cite{MR3177933, MR3527540} the classification of these solutions was reduced to that of braces, and they provided many new classes of these solutions. Recently, skew braces were introduced by L. Guarnieri and L. Vendramin \cite{MR3647970} in order to study the non-degenerate (not necessarily involutive) set-theoretic solutions, and in a subsequent paper the connection of skew braces to ring theory and Hopf-Galois theory was studied by  N. Byott, A. Smoktunowicz, and L. Vendramin \cite{MR3763907}.
	
	On the other hand, in 1969 S. Chase and M. Sweedler \cite{MR0260724} introduced the concept of Hopf-Galois extensions in order to generalise the classical Galois theory. Later, Hopf-Galois theory for separable extensions of fields was studied by C. Greither and B. Pareigis \cite{MR878476}. They showed how to recast the problem of classifying all Hopf-Galois structures on a finite separable extension of fields as a problem in group theory. Many advances relating to the classification of Hopf-Galois structures were made by A. Alabadi, N. Byott \cite{MR1402555, MR2030805, MR2363137, MR3715201}, S. Carnahan, L. Childs \cite{MR1704676}, and T. Kohl \cite{MR1644203}. Recently, some properties of Hopf-Galois structures on a separable field extension of degree $ p^{n} $ were investigated by T. Crespo and M. Salguero \cite{CS}. 
	
	Recently, a fruitful discovery, which was initially noticed by D. Bachiller, revealed a connection between Hopf-Galois theory and skew braces, which linked the classification of Hopf-Galois structures to that of skew braces. 
	
	Despite many efforts both the classification of skew braces and Hopf-Galois structures remain widely open. For example, in \cite{MR2298848} cyclic braces were classified, and in \cite{MR3320237} braces of order $ p^{3} $ were classified. Recently, in \cite{doi:10.1142/S0219498819500336} a method for describing skew braces with non-trivial annihilator was given, and braces of order $ p^{2}q $ have been studied in \cite{CD}. The classification and understanding the structure of skew braces has become more important as they find connections to other areas, for example to concepts in ring theory, see \cite{KSV}, and quantum information \cite{SS}, as well as number theory. Recently, a list of open problems on skew braces has been posed by L. Vendramin \cite{LV}.
	
	To this end, in the author's PhD thesis \cite{KNZ}, an explicit and complete classification of skew braces and Hopf-Galois structures of order $ p^{3} $ for a prime number $ p $ was provided using methods of Hopf-Galois theory. In particular, we independently reproved the results of \cite{MR3320237} on braces of order $ p^{3} $. In this paper, as our main results, we provide a classification for skew braces and Hopf-Galois structures of Heisenberg type for a prime $ p $, which we have chosen to be greater than $ 3 $ for simplicity. However, our methods can be adapted for $ p=2,3 $ as well ($ p=2,3 $ has been treated in the author's PhD thesis). We classify these skew braces and Hopf-Galois structures using some methods of N. Byott \cite{MR2030805} and by conducting a deep study into the holomorph of the Heisenberg group.
	
	Furthermore, we determine the automorphism group of each skew brace that we classify, and as a result we are able to determine the Hopf-Galois structures of Heisenberg type on Galois field extensions of degree $ p^{3} $. In our subsequent two papers we aim to provide our findings relating to the classification of skew braces and Hopf-Galois structures of Extraspecial type (of the type $ C_{p^{2}}\rtimes C_{p} $) in one paper, and skew braces and Hopf-Galois structures of type $ C_{p}^{3} $ in the second paper. These results are currently in the author's PhD thesis \cite{KNZ} Sections $ 4.2 $, $ 4.3 $, and $ 4.5 $. 
	
	We shall begin by providing relevant background information and stating a summary of our main results in the rest of this section. The subsequent sections are devoted to the proofs of our results, and at the end of Section \ref{S4} there is a list of all skew braces classified in this paper. We also determine the \textit{socle} and \textit{annihilator} of these skew braces and show that there are non-trivial skew braces of Heisenberg type with trivial socle and annihilator, so these cannot be described by methods of \cite{doi:10.1142/S0219498819500336}.  
	\subsection{Background}
	A \textit{skew (left) brace} \cite[cf.][]{MR3763907} is a triple $ \left(B,\oplus,\odot\right) $ which consists of a set $ B $ together with two operations $ \oplus $ and $ \odot $ such that $ (B,\oplus) $ and $ (B,\odot) $ are groups (they need not be abelian), and the two operations are related by the \textit{skew brace property}:
	\[a\odot\left(b\oplus c\right)=  \left(a\odot b\right)\ominus a\oplus \left(a\odot c\right) \ \text{for every} \ a,b,c \in B, \]  
	where $ \ominus a $ is the inverse of $ a $ with respect to the operation $ \oplus $. The group $ \left(B,\oplus\right) $ is known as the \textit{additive group} of the skew brace $ \left(B,\oplus,\odot\right) $ and $ \left(B,\odot\right) $ as the \textit{multiplicative group}. A morphism, or a map, between two skew braces  \[ \varphi : \left(B_{1},\oplus_{1},\odot_{1}\right) \longrightarrow  \left(B_{2},\oplus_{2},\odot_{2}\right) \] is a map of sets $ \varphi: B_{1} \longrightarrow B_{2} $ such that the maps 
	\[\varphi : \left(B_{1},\oplus_{1}\right) \longrightarrow  (B_{2},\oplus_{2}) \ \text{and} \ \varphi : \left(B_{1},\odot_{1}\right) \longrightarrow  \left(B_{2},\odot_{2}\right) \]
	are group homomorphisms; the map $ \varphi $ is an isomorphism if it is a bijection. 
	
	We call a skew brace $ \left(B,\oplus,\odot\right) $ such that $ \left(B,\oplus\right) \cong N $ and $ \left(B,\odot\right) \cong G $ a $ G $-skew brace of \textbf{type} $ N $; we refer to the \textit{isomorphism type} of $ \left(B,\odot\right) $ as the \textbf{structure} of the skew brace $ \left(B,\oplus,\odot\right) $.  If $ \oplus $ is abelian, nonabelian respectively, we call $ \left(B,\oplus,\odot\right) $ a skew brace of abelian, nonabelian type respectively. We note that a skew brace of abelian type coincides with the one that was initially defined by W. Rump called a brace (aka a classical brace). Skew braces provide non-degenerate (not necessarily involutive) set-theoretic solutions of the quantum Yang-Baxter equation. The paper of A. Smoktunowicz, and L. Vendramin (also N. Byott)  \cite{MR3763907} provides an excellent introduction to skew braces and their connection to noncommutative algebra, mathematical physics, and other areas. 
	
	Next we recall some definitions and facts relating to Hopf-Galois structures and their connection to skew braces. For $ L/K $ a finite Galois extension of fields with Galois group $ G $, A \textit{Hopf-Galois structure} on $ L/K $ consists of a finite dimensional cocommutative $ K $-\textit{Hopf algebra} $ H $, with an action on $ L $, which makes $ L $ into an $ H $-\textit{Galois extension}, i.e., $ H $ acts on $ L $ in such way that the $ K $-module homomorphism
	\begin{align*}
	j:L\otimes_{K} H \longrightarrow \mathrm{End}_{K}(L) \ \text{given by} \ j(x\otimes y)(z)=xy(z) \ \text{for} \ x,z \in L, y\in H
	\end{align*}
	is an isomorphism. For example, the group algebra $ K[G] $ endows $ L/K $ with the \textit{classical Hopf-Galois structure}. However, in general there can be more than one Hopf-Galois structure on $ L/K $. Hopf-Galois structures have applications in Galois module theory; for example, when studying the freeness of rings of integers of extensions of global or local fields as modules (e.g., see \cite{MR1879021}). In 1987, the classification of Hopf-Galois structures was reduced to a group theoretic problem by C. Greither and B. Pareigis \cite{MR878476} via the following theorem.
	\begin{theorem}[C. Greither and B. Pareigis]\label{T1}
		Hopf-Galois structures on $ L/K $ correspond bijectively to regular subgroups $ N \subseteq \mathrm{Perm}(G) $ which are normalised by the image of $ G $, as left translations, inside $ \mathrm{Perm}(G) $. 
	\end{theorem}
	In particular, every $ K $-Hopf algebra $ H $ which endows $ L/K $ with a Hopf-Galois structure is of the form $ L[N]^{G} $ for some $ N \subseteq \mathrm{Perm}(G) $ a regular subgroup normalised by the image of $ G $, as left translations, inside $ \mathrm{Perm}(G) $. Here $ G $ acts on the group algebra $ L[N] $ through its action on $ L $ as field automorphism and on $ N $ by conjugation inside $ \mathrm{Perm}(G) $. Subsequently, the \textit{isomorphism type} of $ N $ became known as the \textbf{type} of the Hopf-Galois structure, and we shall refer to the cardinality of $ N $, which is the same as the degree of the extension $ L/K $, as the \textbf{order} of the Hopf-Galois structure.
	
	The connection between Hopf-Galois structures and braces was initially noticed by D. Bachiller, later this connection was made more explicit by N. Byott and L. Vendramin in \cite{MR3763907}. For example, one can prove (see Section \ref{S2}) that given a $ G $-skew brace $ \left(B,\oplus,\odot\right) $, the map 
	\begin{align*}
	d: \left(B,\oplus\right)& \longmapsto \mathrm{Perm}\left(B,\odot\right)\\
	a&\longmapsto \left(d_{a}: b\longmapsto a\oplus b\right) \ \text{for all} \ a,b \in B
	\end{align*}  
	is a regular embedding, i.e., $ d $ is an injective map whose image $ \Ima d $ is a regular subgroup. In particular, $ \Ima d $ is normalised by the image of $ \left(B,\odot\right) $ in $ \mathrm{Perm}\left(B,\odot\right) $. This together with Theorem \ref{T1} enables us to obtain a Hopf-Galois structure on $ L/K $. Conversely, one always obtains a skew brace from a Hopf-Galois structure. However, there are more Hopf-Galois structures than skew braces, in particular skew braces parametrise Hopf-Galois structures.     
	
	Finally, we remark that since working with $ \mathrm{Perm}(G) $ can often be difficult, as it becomes rapidly large as size of $ G $ increases, in order to overcome this, N. Byott \cite{MR1402555} proves the following statement -- here L. Childs reformulation cf. \cite[p.~57, (7.3) Theorem (Byott)]{MR1767499} is given.
	\begin{theorem}[N. Byott]\label{THG2}
		Let $ N $ be a group. Then there is a bijection between the sets \[\mathcal{N}\stackrel{\mathrm{def}}{=} \left\{\alpha:N\hookrightarrow \mathrm{Perm}(G)\mid \alpha(N) \ \mathrm{is \ regular\ on}\ G\right\} \  \text{and} \] 
		\[\mathcal{G}\stackrel{\mathrm{def}}{=}\left\{\beta:G \hookrightarrow \mathrm{Perm}(N)\mid  \beta(G) \ \mathrm{is\ regular\ on} \ N \right\}.\]
		Under this bijection, if $ \alpha,\alpha'\in \mathcal{N} $ correspond to $ \beta,\beta' \in \mathcal{G} $, then $ \alpha(N)=\alpha'(N) $ if and only if $ \beta(G) $ and $ \beta'(G) $ are conjugate by an element of $ \mathrm{Aut}(N) $. Furthermore, $ \alpha(N) $ is normalised by the left translation if and only if $ \beta(G) $ is contained in $ \mathrm{Hol}(N) $. 
	\end{theorem}
	Using Theorem \ref{THG2}, N. Byott shows that if $ e'(G,N) $ is the number of regular subgroups of $ \mathrm{Hol}(N) $ isomorphic to $ G $, then the number of Hopf-Galois structures on $ L/K $ of type $ N $ is given by 
	\begin{align}\label{E106}
	e(G,N)=\frac{\left\lvert\mathrm{Aut}(G)\right\rvert}{\left\lvert\mathrm{Aut}(N)\right\rvert}e'(G,N).
	\end{align}
	In the author's thesis \cite{KNZ} we used formula (\ref{E106}) to find the number of Hopf-Galois structures, but in the current paper we parametrise Hopf-Galois structures along skew braces and count them using the orbit stabiliser theorem (we obtain the same results, but in the process we additionally find the automorphism groups of our skew braces too).  
 	  
	\subsection{Summary of the main results}
	We give a summary of our main results in this subsection. For the rest of this paper we shall assume $ p>3 $ is a prime number. We shall denote by $ C_{p^{r}} $ the cyclic group of order $ p^{r} $ for any natural number $ r $. 
	
	Recall there are two nonabelian groups of order $ p^{3} $: the exponent $ p $ nonabelian group of order $ p^{3} $, or otherwise known as the Heisenberg group,
	\[   M_{1} \stackrel{\mathrm{def}}{=}\left\langle \rho,\sigma, \tau \mid \rho^{p}=\sigma^{p}=\tau^{p}=1,  \ \sigma\rho=\rho\sigma , \ \tau\rho=\rho\tau, \ \tau\sigma=\rho\sigma\tau \right\rangle \cong C_{p}^{2}\rtimes C_{p}, \] and the exponent $ p^{2} $ nonabelian group of order $ p^{3} $, or otherwise known as the Extraspecial group of order $ p^{3} $, 
	\[   M_{2} \stackrel{\mathrm{def}}{=}\left\langle \sigma, \tau \mid \sigma^{p^{2}}=\tau^{p}=1, \ \tau\sigma=\sigma^{p+1}\tau\right\rangle\cong C_{p^{2}}\rtimes C_{p}. \]   
	In this paper we are concerned with $ M_{1} $. We fix as our type the group $ M_{1} $ and find all skew braces and Hopf-Galois structures of type $ M_{1} $.  The main results of this paper can be summarised as follows.
	\begin{theorem}\label{T2}
		The skew braces of $ M_{1} $ type are precisely  \[ 2p^{2}-p+3 \] $ M_{1} $-braces and \[  2p+1 \]  $ C_{p}^{3} $-braces.
	\end{theorem}
\begin{proof}
	Follows from adding the numbers found in Lemmas \ref{L16}, \ref{L17}, \ref{L18} of Section \ref{S4}, see Proposition \ref{P5}.
\end{proof}
	\begin{theorem}\label{T3}
		Let $ L/K $ be an $ M_{1} $ extension of fields. Then there are \[(2p^{3}-3p+1)p^{2}\] Hopf-Galois structures of $ M_{1} $ type. Let $ L/K $ be a $ C_{p}^{3} $ extension of fields. Then there are \[(p^{3}-1)(p^{2}+p-1)p^{2}\] Hopf-Galois structures of $ M_{1} $ type.
	\end{theorem}
\begin{proof}
	Follows from adding the numbers found in Lemmas \ref{L16H}, \ref{L17H}, \ref{L18} of Section \ref{S4} see Proposition \ref{P5}.
\end{proof}
\section{Preliminaries}\label{S2}
In this section we provide some preliminaries and describe our strategy for classifying skew braces and Hopf-Galois structures. Unless otherwise stated we shall always assume $ G $ and $ N $ are finite groups.
\subsection{Skew braces and Hopf-Galois structures}\label{SB1}
The following proposition provides and explicit connection between Hopf-Galois structures and skew braces (where ideas of the proof are similar to \cite[Proposition A.3]{MR3763907}).
\begin{proposition}[Skew braces and Hopf-Galois structures correspondence]\label{P1}
	There exists a bijective correspondence between isomorphism classes of $ G $-skew braces and classes of Hopf-Galois structures on an extension $ L/K $ with Galois group $ G $, where we identify two Hopf algebras $ L[N_{1}]^{G} $ and $ L[N_{2}]^{G} $ giving Hopf-Galois structures (as in Theorem \ref{T1}) on $ L/K $ if $ N_{2}=\alpha N_{1}\alpha^{-1} $ for some $ \alpha \in \mathrm{Aut}(G) $.
\end{proposition}
\begin{proof}
	Let $ \left(B,\oplus, \odot\right) $ be a $ G $-skew brace i.e., $ \left(B, \odot\right)\cong G $, we can assume $ \left(B, \odot\right)= G $. Then the map 
	\begin{align*}
	d: \left(B,\oplus\right)& \longrightarrow \mathrm{Perm}\left(B,\odot\right)\\
	a&\longmapsto \left(d_{a}: b\longmapsto a\oplus b\right) \ \text{for all} \ a,b \in B
	\end{align*}  
	is a regular embedding. Now, for any $ a \in \left(B,\oplus\right)  $ and $ b,c \in \left(B,\odot\right)  $, using the skew brace property, we have  
	\[b\odot \left(d_{a}\left(b^{-1}\odot c\right)\right)=b\odot \left(a\oplus\left(b^{-1}\odot c\right)\right)= \left(\left(b\odot a\right) \ominus b\right) \oplus c=d_{\left(b\odot a\right) \ominus b}(c),  \]
	where $ b^{-1} $ is the inverse of $ b $ with respect to $ \odot $. This shows that the image of $ \left(B,\oplus\right) $ is normalised by the image of $ \left(B, \odot\right) $ inside $ \mathrm{Perm}\left(B, \odot\right) $ as left translations. We also find an action of $ \left(B, \odot\right) $ on $ \left(B, \oplus\right) $ by $ b\cdot a= \left(b\odot a\right) \ominus b $ for $ b \in \left(B, \odot\right)  $ and $ a \in \left(B, \oplus\right) $. Now for  \[ \alpha: \left(B,\oplus_{1},\odot\right)\longrightarrow   \left(B,\oplus_{2},\odot\right)  \] an isomorphism of skew braces, we have a commutative diagram
	\[ 
	\begin{tikzcd}[row sep=2.5em , column sep=2.5em]
	\left(B,\oplus_{1}\right) \arrow[hook]{r}{d_{1}} \arrow{d}{\alpha}[swap]{\wr} & \mathrm{Perm}\left(B,\odot\right)  \arrow[]{d}{C_{\alpha}}[swap]{\wr}  \\ \left(B,\oplus_{2}\right)  \arrow[hook]{r}{d_{2}}[swap]{} & \mathrm{Perm}\left(B,\odot\right) , 
	\end{tikzcd} \]
	where $ C_{\alpha} $ is conjugation by $ \alpha \in \mathrm{Aut}\left(B,\odot\right) $ inside $ \mathrm{Perm}\left(B,\odot\right) $. Furthermore, if we fix a Galois extension of fields $ L/K $ with Galois group $ \left(B, \odot\right) $, then $ L[\left(B, \oplus\right)]^{\left(B, \odot\right)} $ endows $ L/K $ with a Hopf-Galois structure corresponding to the skew brace $ \left(B,\oplus, \odot\right) $ and when two skew braces with the same multiplication group are isomorphic then the corresponding Hopf-Galois structures can be identified.  
	
	Conversely, suppose we have a Hopf-Galois structure on $ L/K $ which can always be given by $ L[N]^{G} $ for some regular subgroup $ N \subseteq \mathrm{Perm}(G) $ which is normalised by the image of $ G $ as left translations inside $ \mathrm{Perm}(G) $. The fact that $ N $ is a regular subgroup implies that the map 
	\begin{align*}
	\mathrm{Perm}(G)&\longrightarrow G \\
	\eta&\longmapsto \eta\cdot 1_{G}.
	\end{align*}
	induces a bijection $ \phi:N\longrightarrow G $ as subgroups of $ \mathrm{Perm}(G) $. Now we can define a skew brace $ B $ by setting $ (B,\odot)\stackrel{\mathrm{def}}{=}G $, considered as a subgroup of $ \mathrm{Perm}(G) $ via the left translations, and defining 
	\[g_{1}\oplus g_{2} \stackrel{\mathrm{def}}{=}\phi \left(\phi^{-1}(g_{1})\phi^{-1}(g_{2})\right) \ \text{for} \ g_{1},g_{2} \in G. \]
	The fact that $ N \subseteq \mathrm{Perm}(G) $ is normalised by $ G $ implies that for all $ g \in G $ and $ n \in N $ we have $ gn=f_{g,n}g $ for some $ f_{g,n} \in N $. Therefore, for $ g_{1}=\phi(n_{1}),g_{2}=\phi(n_{2}),g_{3}=\phi(n_{3}) \in G $, we aim to show
	\[ g_{1}\odot \left(g_{2}\oplus g_{3}\right)= (g_{1}\odot g_{2})\ominus g_{1}\oplus( g_{1}\odot g_{3}).\]
	By definitions above we have 
	\[g_{1}\odot \left(g_{2}\oplus g_{3}\right)=\phi(n_{1})\odot \left(\phi(n_{2})\oplus \phi(n_{3})\right)= \phi(n_{1})\phi(n_{2}n_{3}).\]
	Now consider the element $ \phi(n_{1})n_{2}n_{3} \in \mathrm{Perm}(G) $. Using the relation $ gn=f_{g,n}g $, we have 
	\[\phi(n_{1})n_{2}n_{3}=f_{\phi(n_{1}),n_{2}n_{3}}\phi(n_{1}) \]
	for some $ f_{\phi(n_{1}),n_{2}n_{3}}\in N $. Now applying $ \phi $ to both side we get the relation 
	\[\phi(n_{1})\phi(n_{2}n_{3})=f_{\phi(n_{1}),n_{2}n_{3}}(\phi(n_{1}))\]
	in $ G $. Note $ f_{\phi(n_{1}),n_{2}n_{3}}(\phi(n_{1}))=\phi\left(f_{\phi(n_{1}), n_{2}n_{3}}n_{1}\right) $ in $ G $. Therefore, we find
	\begin{align*}
	g_{1}\odot \left(g_{2}\oplus g_{3}\right)&=\phi\left(f_{\phi(n_{1}), n_{2}n_{3}}n_{1}\right)=\phi\left(f_{\phi(n_{1}), n_{2}}f_{\phi(n_{1}), n_{3}}n_{1}\right)\\
	&=\phi\left(\phi^{-1}\phi \left(f_{\phi(n_{1}), n_{2}}n_{1}\right)n_{1}^{-1}\phi^{-1}\phi\left(f_{\phi(n_{1}), n_{3}}n_{1}\right)\right)\\
	&=\phi\left(\phi^{-1}\left(\phi(n_{1})\phi(n_{2})\right)n_{1}^{-1}\phi^{-1}\left(\phi(n_{1})\phi(n_{3})\right)\right)\\
	&=\phi\left(\phi^{-1}\left(g_{1}g_{2}\right)\left(\phi^{-1}(g_{1})\right)^{-1}\phi^{-1}\left(g_{1}g_{3}\right)\right)\\
	&=(g_{1}\odot g_{2})\ominus g_{1}\oplus( g_{1}\odot g_{3});
	\end{align*} 
	thus we have a skew brace $ (B,\oplus, \odot) $ which is a $ G $-skew brace of type $ N $. In particular, if $  N_{1} \subseteq \mathrm{Perm}(G)  $ is a regular subgroups whose image is normalised by $ G $ and $ \alpha \in \mathrm{Aut}(G) $, then $ N_{2}\stackrel{\mathrm{def}}{=} \alpha N_{1}\alpha^{-1} $ is a regular subgroup whose image is normalised by $ G $ and the skew braces corresponding to $ N_{1} $ and $ N_{2} $ are isomorphic by $ \alpha $.
\end{proof}
\begin{remark}
	Note in fact Proposition \ref{P1} above is implied by Theorem \ref{THG2} and \cite[Proposition A.3]{MR3763907}. We shall state \cite[Proposition A.3]{MR3763907} later (see Proposition \ref{P2}). However, we decided to include the calculations for a direct proof of Proposition \ref{P1} for completeness, which leads to an explicit relationship between the Hopf-Galois structures and skew braces. The question relating to the explicit relationship between the Hopf-Galois structures and skew braces was first asked from the author by Prof Agata Smoktunowicz. The answer can be reached by unravelling Theorem \ref{THG2} and \cite[Proposition A.3]{MR3763907} which is what has been done in Proposition \ref{P1}.
\end{remark}

The above proposition also helps us to understand the automorphism groups of skew braces. 
\begin{corollary}[Automorphism groups of skew braces]
	Let $ \left(B,\oplus,\odot\right) $ be a skew brace. Then there exists a natural identification 
	\[ \mathrm{Aut}_{\mathcal{B}r}\left(B,\oplus,\odot\right)\cong \left\{ \alpha \in \mathrm{Aut}\left(B,\odot\right)\mid \alpha \left(\Ima d\right) \alpha ^{-1} \subseteq \Ima d \right\}. \] 
\end{corollary}
\begin{proof}
Note that if $ \left(B,\oplus,\odot\right) $ is a skew brace and  \[ \alpha: \left(B,\oplus,\odot\right)\longrightarrow   \left(B,\oplus,\odot\right)  \] an automorphism of skew braces, we have a commutative diagram
\[ 
\begin{tikzcd}[row sep=2.5em , column sep=2.5em]
\left(B,\oplus\right) \arrow[hook]{r}{d} \arrow{d}{\alpha}[swap]{\wr} & \mathrm{Perm}\left(B,\odot\right)  \arrow[]{d}{C_{\alpha}}[swap]{\wr}  \\ \left(B,\oplus\right)  \arrow[hook]{r}{d}[swap]{} & \mathrm{Perm}\left(B,\odot\right), 
\end{tikzcd} \]
implying that $ \alpha\left(\Ima d \right)\alpha^{-1} \subseteq \Ima d $. On the other hand, if $ \alpha\left(\Ima d \right)\alpha^{-1} \subseteq \Ima d $ for some $ \alpha \in  \mathrm{Aut}\left(B,\odot\right) $, then $ \alpha $ gives automorphism of $ \left(B,\oplus,\odot\right) $. From this observation one can see that 
\[ \mathrm{Aut}_{\mathcal{B}r}\left(B,\oplus,\odot\right)\cong \left\{ \alpha \in \mathrm{Aut}\left(B,\odot\right)\mid \alpha \left(\Ima d\right) \alpha ^{-1} \subseteq \Ima d \right\}. \] 	
\end{proof}
Next corollary shows how to obtain the number of Hopf-Galois structures using skew braces. Let $ e(G,N) $ be the number of Hopf-Galois structures of type $ N $ on the field extension $ L/K $ whose Galois group is $ G $. Denote by $ G_{N} $ the isomorphism class of a $ G $-skew brace of type $ N $. For later use we also set $  \widetilde{e}(G,N)  $ to be the number of isomorphism classes of $ G $-skew braces of type $ N $.
\begin{corollary}[Number of Hopf-Galois structures parametrised by skew braces]
	We have 
	\begin{align}\label{E1}
	e(G,N)= \sum_{G_{N}}\dfrac{\left\lvert \mathrm{Aut}(G)\right\rvert} {\left\lvert\mathrm{Aut}_{\mathcal{B}r}(G_{N})\right\rvert}.
	\end{align}
\end{corollary}
\begin{proof}
	Fix $ G $ and let \[ \mathcal{S}(G,N) = \left\{ M \subseteq \mathrm{Perm}(G) \mid M \cong N \ \text{and} \ M\ \text{is regular normalised by}\ G \right\}. \] 
	Firstly, note that $ \mathrm{Aut}(G) $ acts on $ \mathcal{S}(G,N) $, induced by conjugation in $ \mathrm{Perm}(G) $, and a set of orbit representatives, say $ \left\{ N_{1},...,N_{s} \right\} $, give a list of non-isomorphic skew braces according to Proposition \ref{P1}. Secondly, by Theorem \ref{T1} we find $ e(G,N)=\left\lvert \mathcal{S}(G,N) \right \rvert $, and so we have
	\begin{align*}
	e(G,N)=\sum_{i=1}^{s}\left\lvert \mathrm{Orb}(N_{i}) \right \rvert = \sum_{i=1}^{s} \dfrac{\left\lvert \mathrm{Aut}(G)\right\rvert}{\left\lvert\mathrm{Stab}(N_{i})\right\rvert}= \sum_{G_{N}}\dfrac{\left\lvert \mathrm{Aut}(G)\right\rvert} {\left\lvert\mathrm{Aut}_{\mathcal{B}r}(G_{N})\right\rvert}.
	\end{align*}
\end{proof}
Therefore, to find skew braces and Hopf-Galois structures of order $ n $, one can find the regular subgroups $ N \subseteq \mathrm{Perm}(G) $ for every group $ G $ of size $ n $. However, in many cases $ \mathrm{Perm}(G) $ can be too large to handle. Fortunately, by somehow reversing the role of $ G $ and $ N $, instead of studying the regular subgroups of $ \mathrm{Perm}(G) $, one can study regular subgroups of a smaller group, the \textit{holomorph} of $ N $: \[ \mathrm{Hol}(N)\stackrel{\mathrm{def}}{=}N\rtimes \mathrm{Aut}(N)=\left\{\eta\alpha  \mid \eta \in N, \ \alpha \in \mathrm{Aut}(N) \right\},\]
also we can organise these objects in a nice way. These ideas in Hopf-Galois theory were initially developed by N. Byott \cite{MR1402555, MR2030805}. 

For skew braces we observe the following.  Let $  \left(B,\oplus,\odot\right) $ be a skew brace. Then the group $ \left(B,\odot\right) $ acts on $ \left(B,\oplus\right) $ by $ (a,b)\longmapsto a\odot b $, and we obtain a map
\begin{align*}
m: \left(B,\odot\right) &\longrightarrow \mathrm{Hol}\left(B,\oplus\right)\\
a&\longmapsto \left(m_{a} : b \longmapsto a\odot b\right)
\end{align*} 
which is a regular embedding. To see this one needs to check that the map 
\begin{align*}
\lambda_{a}: \left(B,\oplus\right) &\longrightarrow \left(B,\oplus\right)\\
b&\longmapsto \ominus a \oplus (a\odot b)
\end{align*} 
is an automorphism, and that the map
\begin{align*}
\lambda: \left(B,\odot\right)&\longmapsto  \mathrm{Aut}\left(B,\oplus\right)\\
a&\longmapsto \lambda_{a}
\end{align*} 
is a group homomorphism. Then one has $ m_{a}=a\lambda_{a}\in \mathrm{Hol}\left(B,\oplus\right) $ for all $ a \in B $. Additionally, for  $ \alpha: \left(B,\oplus,\odot_{1}\right)\longrightarrow   \left(B,\oplus,\odot_{2}\right)  $ an isomorphism of skew braces, we have
\[ 
\begin{tikzcd}[row sep=2.5em , column sep=2.5em]
\left(B,\odot_{1}\right) \arrow[hook]{r}{m_{1}} \arrow{d}{\alpha}[swap]{\wr} & \mathrm{Hol}\left(B,\oplus\right)  \arrow[]{d}{C_{\alpha}}[swap]{\wr} \\ \left(B,\odot_{2}\right)  \arrow[hook]{r}{m_{2}}[swap]{} & \mathrm{Hol}\left(B,\oplus\right),
\end{tikzcd} \]
where $ C_{\alpha} $ is conjugation by $ \alpha \in \mathrm{Aut}\left(B,\oplus\right) $ considered naturally as an element of $ \mathrm{Hol}\left(B,\oplus\right) $. This with similar procedure as used to prove Proposition \ref{P1} gives the following proposition of \cite{MR3763907}. 
\begin{proposition}\label{P2}
	There exists a bijective correspondence between isomorphism classes of skew braces of type $ N $ and classes of regular subgroups of $ \mathrm{Hol}(N) $ under conjugation by elements
	of $ \mathrm{Aut}(N) $.
\end{proposition}
\begin{proof}
	\cite[Proposition A.3]{MR3763907}.
\end{proof}
In particular, we find another way of computing automorphism groups of skew braces:
\begin{align}\label{E2}
 \mathrm{Aut}_{\mathcal{B}r}\left(B,\oplus,\odot\right)\cong \left\{ \alpha \in \mathrm{Aut}\left(B,\oplus\right)\mid \alpha \left(\Ima m\right) \alpha ^{-1} \subseteq \Ima m \right\}. 
\end{align}

Therefore, in this way to find the set of non-isomorphic $ G $-skew braces of type $ N $, it suffices to find the set of regular subgroups of $ \mathrm{Hol}(N) $ which are isomorphic to $ G $, and then extract a maximal subset whose elements are not conjugate by any element of $ \mathrm{Aut}(N) $. In particular, (cf. \cite{MR2030805}) one can organise these regular subgroups, and hence the corresponding skew braces and Hopf-Galois structures, according to the size of their image under the natural projection 
\begin{align}\label{E4}
\varTheta : \mathrm{Hol}(N) &\longrightarrow \mathrm{Aut}(N)\nonumber\\
\eta\alpha&\longmapsto \alpha.
\end{align}
In other words, if $ \widetilde{\mathcal{S}}(G,N,r) $ is the set of regular subgroups of $ \mathrm{Hol}(N) $ isomorphic to $ G $ whose image under the natural projection $ \varTheta $ has size $ r $, then the set of regular subgroups of $ \mathrm{Hol}(N) $ isomorphic to $ G $ is a finite disjoint union
\[\widetilde{\mathcal{S}}(G,N)= \coprod_{r}\widetilde{\mathcal{S}}(G,N,r). \]
Furthermore, $ \mathrm{Aut}(N) $ acts on each $ \widetilde{\mathcal{S}}(G,N,r) $ via conjugation inside $ \mathrm{Hol}(N) $, and a set of orbit representatives provides a set of isomorphism classes of $ G $-skew brace of type $ N $, whose size upon embedding in $ \mathrm{Hol}(N) $ and projecting to $ \mathrm{Aut}(N) $ is $ r $, which we shall denote by $ G_{N}(r) $. In order to find the number of Hopf-Galois structures of type $ N $ it suffices to find the automorphism group of each $ G $-skew braces of type $ N $ using (\ref{E2}) and use the formula given in (\ref{E1}). We shall set $ e'(G,N,r)=\lvert \widetilde{\mathcal{S}}(G,N,r) \rvert $ and denote by $ \widetilde{e}(G,N,r) $ the number of isomorphism classes of skew braces $ G_{N}(r) $.

\subsection{Regular subgroups of holomorphs}\label{SB2} 
In this subsection we outline our strategy for finding regular subgroups contained in $ \mathrm{Hol}(N) $. Let us denote by \[ \varTheta: \mathrm{Hol}(N) \longrightarrow \mathrm{Aut}(N),\]
the natural projection with kernel $ N $. Then the first step is to organise the regular subgroups of $ \mathrm{Hol}(N) $ according to the size of their image under the map $ \varTheta $.

Now suppose we want to parametrise subgroups $ H\subseteq \mathrm{Hol}(N) $ with $ \left\lvert \varTheta(H)\right\rvert=m $, where $ m $ divides $\left\lvert N \right\rvert $. In order to do this, we first take a subgroup of order $ m $ of $ \mathrm{Aut}(N) $, which may be generated by some elements $ \alpha_{1},...,\alpha_{s} \in \mathrm{Aut}(N) $, say \[ H_{2}\stackrel{\mathrm{def}}{=} \left\langle  \alpha_{1},...,\alpha_{s} \right\rangle \subseteq \mathrm{Aut}(N).\] Next, we take a subgroup of order $ \frac{\left\lvert N \right\rvert}{m} $ of $ N $, which may be generated by $ \eta_{1},...,\eta_{r} \in N $, say \[ H_{1}\stackrel{\mathrm{def}}{=}\left\langle \eta_{1},...,\eta_{r} \right\rangle \subseteq N.\] We also take `general elements' $ v_{1},...,v_{s} \in N $, and we consider a subgroup of $ \mathrm{Hol}(N) $ of the form 
\[H=\left\langle \eta_{1},...,\eta_{r},v_{1}\alpha_{1},...,v_{s}\alpha_{s} \right\rangle. \]
Now we need to classify the constraints on $ v_{1},...,v_{s} $ such that $ H $ is regular, i.e., $ H $ has the same size as $ N $ and acts freely on $ N $. It is easy to see that there are many restrictions on $ v_{1},...,v_{s} $ and in many cases no choice of $ v_{1},...,v_{s} $ will result is a regular subgroup.

Notice that $ \left\lvert H \right\rvert  \geq\left\lvert N \right\rvert $ since we have the following commutative diagram 
\[ 
\begin{tikzcd}[row sep=2.5em , column sep=2.5em]
 H_{1} \arrow[hook]{r}{} \arrow[hook]{d}{}[swap]{} & H \arrow[two heads]{r}{\varTheta} \arrow[hook]{d}{}[swap]{} & H_{2} \arrow[hook]{d}{}[swap]{} \\ N  \arrow[hook]{r}{}[swap]{} &  \mathrm{Hol}(N) \arrow[two heads]{r}{\varTheta} & \mathrm{Aut}(N) ,
\end{tikzcd} \]
where the hook arrows are natural inclusion, and the second row is exact, but the first row is not necessarily exact. One of our goals is to select $ v_{1},...,v_{s} $ such that the first row is exact, which would implies that $ \left\lvert H \right\rvert = \left\lvert N \right\rvert $. In particular, we need $ H\cap N=H_{1} $.  That is for example, if there is a relation say $ \alpha^{a_{1}}=1 $ in $ H_{2} $, then we need to ensure that $ \left(v_{1}a_{1} \right)^{a_{1}}=v_{1}v_{1}^{\alpha_{1}}\cdots v_{1}^{\alpha_{1}^{a_{1}-1}}\in H_{1} $. Furthermore, we need to ensure that $ H $ acts freely on $ N $, and so for example, if $ v_{i} \in H_{1} $ for some $ i $, then $ H $ will not be acting freely.  

More generally we require the following. For $ H $ to have the same size as $ N $, we require for every relation $ R\left(\alpha_{1},...,\alpha_{s}\right)=1 $ on $ H_{2} $ to have \[R\left(u_{1}\left(v_{1}\alpha_{1}\right)w_{1},...,u_{s}\left(v_{s}\alpha_{s}\right)w_{s}\right)\in H_{1}, \] 
for every $ u_{1},w_{1},...,u_{s},w_{s}\in H_{1} $. For $ H $ to act freely on $ N $, it is necessary that for every word $ W\left(\alpha_{1},...,\alpha_{s}\right)\neq1 $ on $ H_{2} $ we require 
\[W(u_{1}\left(v_{1}\alpha_{1})w_{1},...,u_{s}(v_{s}\alpha_{s})w_{s}\right)W\left(\alpha_{1},...,\alpha_{s}\right)^{-1} \notin H_{1}, \] 
for every $ u_{1},w_{1},...,u_{s},w_{s}\in H_{1} $; so in fact we must have \[\left\langle \eta_{1},...,\eta_{r},v_{1},...,v_{s} \right\rangle = N.  \] 
However, in general there may be other conditions on $ v_{i} $ that need to  be taken into account -- for example, some elements of $ H $ need to satisfy relations between generators of a group of order $ \left\lvert N \right\rvert $. Therefore, as already mentioned, it can happen that desirable $ v_{i} $ cannot be found. To find all regular subgroups we repeat this process for every $ m $, every subgroup of order $ m $ of $ \mathrm{Aut}(N) $, and every subgroup of order $ \frac{\left\lvert N \right\rvert}{m} $ of $ N $.

Finally, in order to find non-isomorphic skew braces, we need to check which of these regular subgroups are conjugate to one another by elements of $ \mathrm{Aut}(N) $. Note, if $ H $ and $ \widetilde{H} $ are regular subgroups of $ \mathrm{Hol}(N) $ with $ \lvert\varTheta(H)\rvert=\lvert\varTheta( \widetilde{H})\rvert=m $, then $ H $ and $ \widetilde{H} $ are conjugate by an element of $ \beta \in \mathrm{Aut}(N) $ if 
\[\beta(H_{1})\subseteq \widetilde{H}_{1} \ \text{and} \ \beta H_{2}\beta^{-1}\subseteq \widetilde{H}_{2}, \] 
i.e., when $ H=\left\langle \eta_{1},...,\eta_{r},v_{1}\alpha_{1},...,v_{s}\alpha_{s} \right\rangle $, we need 
\[ \left\langle \eta_{1}^{\beta},...,\eta_{r}^{\beta},v_{1}^{\beta}\beta\alpha_{1}\beta^{-1},...,v_{s}^{\beta}\beta\alpha_{s}\beta^{-1} \right\rangle \subseteq \widetilde{H}.\]

Our starting point is studying the Heisenberg group of order $ p^{3} $ and its automorphism group. 
\section{The Heisenberg group $ M_{1} $}\label{S3}
For $ p>2 $ the exponent $ p $ nonabelian group of order $ p^{3} $, or otherwise known as the Heisenberg group, which we denote by $ M_{1} $,  has a presentation \[   M_{1} \stackrel{\mathrm{def}}{=}\left\langle \rho,\sigma, \tau \mid \rho^{p}=\sigma^{p}=\tau^{p}=1, \ \sigma\rho=\rho\sigma,  \ \tau\rho=\rho\tau, \ \tau\sigma=\rho\sigma\tau\right\rangle \cong C_{p}^{2}\rtimes C_{p}. \]
Note, the above relations imply that for positive integers $ a_{1},a_{2},a_{3},a_{4} $, we have \[\sigma^{a_{1}}\tau^{a_{2}}\sigma^{a_{3}}\tau^{a_{4}}=\rho^{a_{2}a_{3}}\sigma^{a_{1}+a_{3}}\tau^{a_{2}+a_{4}}\]
from which we also obtain the relation
\begin{align}\label{GE1}
(\sigma^{a_{1}}\tau^{a_{2}})^{n}=\rho^{\frac{1}{2}a_{1}a_{2}n(n-1)}\sigma^{na_{1}}\tau^{na_{2}}.
\end{align}

We note that the group $ M_{1} $ contains $ p^{3}-1 $ elements of order $ p $, thus $ p^{2}+p+1 $ subgroups of order $ p $, which are of the form \[ \left\langle\rho\right\rangle,\left\langle\rho^{a}\sigma\right\rangle, \left\langle\rho^{b}\sigma^{c}\tau\right\rangle \ \text{for} \ a,b,c=0,...,p-1. \] Also $ M_{1} $ contains $ p+1 $ subgroups of order $ p^{2} $, which are all isomorphic to $ C_{p}^{2} $, of the form \[ \left\langle\rho,\tau\right\rangle, \left\langle\rho,\sigma\tau^{d}\right\rangle \ \text{for} \ d=0,...,p-1.\]
The next proposition determines the automorphism group of $ M_{1} $. For the analogous result over $ \mathbb{Z} $ see \cite{DVO}. I am grateful to the referee for drawing my attention to this reference.
\begin{proposition}\label{P3}
	We have $ \left\lvert \mathrm{Aut}(M_{1})\right\rvert =(p^{2}-1)(p-1)p^{3} $ and 
	\[\mathrm{Aut}(M_{1})\cong C_{p}^{2}\rtimes \mathrm{GL}_{2}(\mathbb{F}_{p}),  \]
	where $ C_{p}^{2} $ in the semi-direct product above is generated by the automorphisms $ \beta, \gamma \in \mathrm{Aut}(M_{1}) $ defined by
	\begin{align*}
	\sigma^{\beta}&=\sigma, \ \tau^{\beta}=\rho\tau \ \text{and} \\
	\sigma^{\gamma}&=\rho\sigma, \ \tau^{\gamma}=\tau.
	\end{align*} 
	
	The (left) action of $ \mathrm{GL}_{2}(\mathbb{F}_{p}) $ on $ C_{p}^{2}=\left\langle \beta,\gamma \right\rangle $, in the semi-direct product, is given by \[\begin{pmatrix} a_{1} & a_{2} \\ a_{3} & a_{4} \end{pmatrix}\cdot\beta= \beta^{a_{1}}\gamma^{-a_{3}} \ \text{and} \ \begin{pmatrix} a_{1} & a_{2} \\ a_{3} & a_{4} \end{pmatrix}\cdot\gamma= \beta^{-a_{2}}\gamma^{a_{4}}. \]  
	where $ \begin{pmatrix} a_{1} & a_{2} \\ a_{3} & a_{4} \end{pmatrix} \in \mathrm{GL}_{2}(\mathbb{F}_{p}) $.
\end{proposition}
\begin{proof}
	Let $ \alpha \in  \mathrm{Aut}(M_{1}) $. Then we have
	\begin{align*}
	\sigma^{\alpha}&=\rho^{b_{1}}\sigma^{a_{1}}\tau^{a_{3}} \\
	\tau^{\alpha}&=\rho^{b_{2}}\sigma^{a_{2}}\tau^{a_{4}}
	\end{align*}
	for some $ a_{1},a_{2},a_{3},a_{4},b_{1},b_{2} \in \mathbb{Z} /p\mathbb{Z} $. Note, $ \rho^{\alpha} $ is determined by above and we find \[\rho^{\alpha}=\tau^{\alpha}\sigma^{\alpha}\left(\sigma^{\alpha}\tau^{\alpha}\right)^{-1}=\rho^{a_{1}a_{4}-a_{2}a_{3}},\] so $ \alpha $ is bijective if and only if $ a_{1}a_{4}-a_{2}a_{3}\not\equiv 0 \ \mathrm{mod} \ p $. We shall write \[ \begin{bmatrix} a_{1}a_{4}-a_{2}a_{3} & b_{1} & b_{2} \\ 0 & a_{1} & a_{2} \\ 0 & a_{3} & a_{4} \end{bmatrix}\ \text{or} \ \begin{bmatrix} \mathrm{det}(A) & b_{1} & b_{2} \\ 0 & \ \ \ \ \ A \end{bmatrix} \] to represent $ \alpha $. This is only a representation, and not a matrix, so composition of automorphisms does not in general correspond to matrix multiplication. In fact composition of automorphisms yields the following. 
	\begin{align*}
	&\begin{bmatrix} \mathrm{det}(A) & b_{1} & b_{2} \\ 0 & A \end{bmatrix}\circ \begin{bmatrix} \mathrm{det}(A') & b_{1}' & b_{2}' \\ 0 & \ \ \ \ \ A' \end{bmatrix} \\
	&=\begin{bmatrix} \begin{pmatrix} \mathrm{det}(A) & b_{1} & b_{2} \\ 0 & \ \ \ \ \ A \end{pmatrix} \begin{pmatrix} \mathrm{det}(A') & b_{1}' & b_{2}' \\ 0 & \ \ \ \ \ A' \end{pmatrix} +  \begin{pmatrix} 0 & C_{1} & C_{2} \\ 0 & \ \ \ \ \ 0 \end{pmatrix} \end{bmatrix}
	\end{align*}
	for 
	\begin{align*}
	C_{1}&=\frac{1}{2}a_{1}a_{3}a_{1}'(a_{1}'-1)+\frac{1}{2}a_{2}a_{4}a_{3}'(a_{3}'-1)+a_{3}a_{1}'a_{2}a_{3}'\\
	C_{2}&=\frac{1}{2}a_{1}a_{3}a_{2}'(a_{2}'-1)+\frac{1}{2}a_{2}a_{4}a_{4}'(a_{4}'-1)+a_{3}a_{2}'a_{2}a_{4}'.
	\end{align*}
	
	The group $ M_{1} $ has centre $  Z= \left\langle\rho\right\rangle $ of order $ p $ and \[ M_{1}/Z= \left\langle \overline{\sigma}, \overline{\tau}\right\rangle \cong C_{p}^{2},\] where $ \overline{\sigma},\overline{\tau} \in M_{1}/Z $ are the images of $ \sigma,\tau \in M_{1} $. Thus we obtain a natural homomorphism \[ \varPsi: \mathrm{Aut}(M_{1})\longrightarrow \mathrm{Aut}(M_{1}/Z) \cong \mathrm{GL}_{2}(\mathbb{F}_{p}). \] Since $ M_{1}/Z\cong C_{p}^{2} $ is abelian, we see that the set of inner automorphisms of $ M_{1} $ is contained in the kernel of $ \varPsi $ i.e., $ \mathrm{Inn}(M_{1}) \subseteq \Ker \varPsi $. Note $ \mathrm{Inn}(M_{1}) \cong M_{1}/Z  $. Now if $ \alpha \in \Ker \varPsi $, then we must have $ \tau^{\alpha}\tau^{-1}\in Z $ and $ \sigma^{\alpha}\sigma^{-1}\in Z $ i.e., 
	\begin{align*}
	\sigma^{\alpha}&=\rho^{r_{1}}\sigma \\ 
	\tau^{\alpha}&=\rho^{r_{2}}\tau
	\end{align*}
	for some integers $ r_{1}, r_{2} = 0,...,p-1 $, which implies that $ \rho^{\alpha}=\rho $. There can be at most $ p^{2} $ choices for such $ \alpha $, which implies that $ \mathrm{Inn}(M_{1}) = \Ker \varPsi $. We further find $ \Ker \varPsi=\left\langle\beta,\gamma\right\rangle $ where 
	\[ \beta\stackrel{\mathrm{def}}{=}\begin{bmatrix} 1 & 0 & 1 \\ 0 & 1 & 0 \\ 0 & 0 & 1 \end{bmatrix},\ \gamma\stackrel{\mathrm{def}}{=}\begin{bmatrix} 1 & 1 & 0 \\ 0 & 1 & 0 \\ 0 & 0 & 1 \end{bmatrix}.  \]  
	
	To show that the map $ \varPsi $ is surjective, for any element \[ A\stackrel{\mathrm{def}}{=}\begin{pmatrix} a & b \\ c & d \end{pmatrix} \in \mathrm{GL}_{2}(\mathbb{F}_{p}) \] define a map \[ \alpha_{A} : M_{1} \longrightarrow M_{1} \ \text{given by} \  \alpha_{A}\stackrel{\mathrm{def}}{=}\begin{bmatrix} ad-bc & \frac{ac}{2} & \frac{bd}{2} \\ 0 & a & b \\ 0 & c & d \end{bmatrix}.\]  
	It is also easy to check that $ A\longmapsto \alpha_{A} $ is a group homomorphism. Therefore, we find a split exact sequence  \[ 
	\begin{tikzcd}[row sep=1.1em , column sep=1.1em]
	1 \arrow[]{r}[]{} & C_{p}^{2}  \arrow{r}[]{}  & \mathrm{Aut}(M_{1}) \arrow{r}[]{}   & \mathrm{GL}_{2}(\mathbb{F}_{p}) \arrow{r}[]{} & 1.  
	\end{tikzcd} \] 
	One can check that the left action of $ \mathrm{GL}_{2}(\mathbb{F}_{p}) $ on $ C_{p}^{2} $ is given by \[\begin{pmatrix} a_{1} & a_{2} \\ a_{3} & a_{4} \end{pmatrix}\cdot\beta= \beta^{a_{1}}\gamma^{-a_{3}} \  \text{and} \ \begin{pmatrix} a_{1} & a_{2} \\ a_{3} & a_{4} \end{pmatrix}\cdot\gamma= \beta^{-a_{2}}\gamma^{a_{4}}. \]
	Note the above corresponds \[\alpha_{A}\beta= \beta^{a_{1}}\gamma^{-a_{3}}\alpha_{A} \  \text{and} \ \alpha_{A}\gamma= \beta^{-a_{2}}\gamma^{a_{4}}\alpha_{A}. \]  
\end{proof} 
\section{Skew braces of $ M_{1} $ type}\label{S4}
In this section we classify the skew braces and Hopf-Galois structures of $ M_{1} $ type. The main result of this section is the following (which is a proof of Theorems \ref{T2} and \ref{T3}). Recall, $ \widetilde{e}(G,N) $ is the number of $ G $-skew braces of type $ N $ and $ e(G,N) $ is the number of Hopf-Galois structures on a Galois extension with Galois group $ G $ of type $ N $.
\begin{proposition}\label{P5}
	We have
	\begin{align*}
	\widetilde{e}(M_{1},M_{1})&=2p^{2}-p+3,\\
	\widetilde{e}(C_{p}^{3},M_{1})&=2p+1, 
	\end{align*}
	and $ \widetilde{e}(G ,M_{1})=0 $ for  $ G\ncong M_{1} $ or $ C_{p}^{3} $.
	
	Furthermore, we have 
	\begin{align*}
	e(M_{1},M_{1})&=(2p^{3}-3p+1)p^{2},\\
	e(C_{p}^{3},M_{1})&=(p^{3}-1)(p^{2}+p-1)p^{2}, 
	\end{align*}
	and $ e(G ,M_{1})=0 $ for  $ G\ncong M_{1} $ or $ C_{p}^{3} $. 
\end{proposition} 
\begin{proof}
	This follows from the calculation in the rest of this section, particularly the first part follows by adding the relevant numbers from Lemmas \ref{L16}, \ref{L17}, and \ref{L18}
	\begin{align*}
	\widetilde{e}(M_{1},M_{1})&=1+2(p-1)+(2p-3)p+4=2p^{2}-p+3,\\
	\widetilde{e}(C_{p}^{3},M_{1})&=2+2p-1=2p+1, 
	\end{align*}
	and the second part follows by adding relevant numbers from \ref{L16H}, \ref{L17H}, and \ref{L18}
		\begin{align*}
	e(M_{1},M_{1})&=1+(p^{3}-p^{2}-1)(p+1)+(p^{4}-p^{3}-2p^{2}+2p+1)p+(p^{2}-1)p^{3}\\
	&=(2p^{3}-3p+1)p^{2},\\
	e(C_{p}^{3},M_{1})&=(p^{3}-1)(p+1)p^{2}+(p^{3}-1)(p^{2}-2)p^{2}=(p^{3}-1)(p^{2}+p-1)p^{2}.
	\end{align*}
\end{proof}
We note that at the end of lemmas \ref{L16},  \ref{L17}, and \ref{L18} there are lists of non-isomorphic skew braces together with a description of their automorphism groups. 

Before we begin to prove Lemmas \ref{L16}, \ref{L16H},  \ref{L17}, \ref{L17H}, and \ref{L18}, we need to set up some notations. Let us denote by \[ \alpha_{1}\stackrel{\mathrm{def}}{=}\begin{bmatrix} 1 & 1 & 0 \\ 0 & 1 & 0 \\ 0 & 0 & 1 \end{bmatrix},\ \alpha_{2}\stackrel{\mathrm{def}}{=}\begin{bmatrix} 1 & 0 & 0 \\ 0 & 1 & 0 \\ 0 & 1 & 1 \end{bmatrix}, \ \alpha_{3}\stackrel{\mathrm{def}}{=}\begin{bmatrix} 1 & 0 & 1 \\ 0 & 1 & 0 \\ 0 & 0 & 1  \end{bmatrix}.\] 
Note in Proposition \ref{P3}, we had $ \alpha_{1}=\gamma $ and $ \alpha_{3}=\beta $. Furthermore, we showed that $ \mathrm{Aut}(M_{1}) $ can be written as
\[\mathrm{Aut}(M_{1})\cong C_{p}^{2}\rtimes \mathrm{GL}_{2}(\mathbb{F}_{p}),  \]
where the factor $ C_{p}^{2} $ is generated by automorphisms $ \alpha_{1},\alpha_{3} \in \mathrm{Aut}(M_{1}) $. The (left) action of $ \mathrm{GL}_{2}(\mathbb{F}_{p}) $ on $ C_{p}^{2} $ is given by
\begin{align}\label{E95}
\begin{pmatrix} a_{1} & a_{2} \\ a_{3} & a_{4} \end{pmatrix}\cdot\alpha_{1}= \alpha_{1}^{a_{4}}\alpha_{3}^{-a_{2}}, \ \begin{pmatrix} a_{1} & a_{2} \\ a_{3} & a_{4} \end{pmatrix}\cdot\alpha_{3}= \alpha_{1}^{-a_{3}}\alpha_{3}^{a_{1}}.  
\end{align}

Therefore, the holomorph of $ M_{1} $ can be identified with \[ \mathrm{Hol}(M_{1})\cong M_{1}\rtimes (C_{p}^{2}\rtimes\mathrm{GL}_{2}(\mathbb{F}_{p}) ). \] 
Now the image in $ \mathrm{GL}_{2}(\mathbb{F}_{p}) $ of a subgroup $ G \subseteq \mathrm{Hol}(M_{1}) $ of order $ p^{3} $  under the composition of projections \[\varTheta : \mathrm{Hol}(M_{1})\longrightarrow \mathrm{Aut}(M_{1}) \ \text{and} \ \varPsi :  \mathrm{Aut}(M_{1})\longrightarrow \mathrm{GL}_{2}(\mathbb{F}_{p}) \] must lie in one of the $ p+1 $ Sylow $ p $-subgroup of $ \mathrm{GL}_{2}(\mathbb{F}_{p}) $, which are conjugate to the subgroup generated by $ \beta_{1}\stackrel{\mathrm{def}}{=}\begin{psmallmatrix} 1 & 0 \\ 1 & 1  \end{psmallmatrix} $; thus we have \[ \varTheta(G) \subseteq  \mathrm{A}_{\beta}(M_{1})\stackrel{\mathrm{def}}{=} C_{p}^{2}\rtimes\left\langle\beta\beta_{1}\beta^{-1}\right\rangle \cong M_{1} \ \text{for some} \ \beta \in \mathrm{GL}_{2}(\mathbb{F}_{p}), \] and so any subgroup of $ \mathrm{Hol}(M_{1}) $ of order $ p^{3} $ lies in a subgroup of the form \[ M_{1}\rtimes  \mathrm{A}_{\beta}(M_{1})  \ \text{for some} \ \beta \in \mathrm{GL}_{2}(\mathbb{F}_{p}).\] 
Note, the elements $ \alpha_{1}, \alpha_{2}, \alpha_{3} \in \mathrm{Aut}(M_{1}) $ have order $ p $, and they satisfy 
\begin{align}\label{E109}
\alpha_{2}\alpha_{1}=\alpha_{1}\alpha_{2}, \ \alpha_{3}\alpha_{1}=\alpha_{1}\alpha_{3}, \ \alpha_{3}\alpha_{2}=\alpha_{1}\alpha_{2}\alpha_{3}.
\end{align} 
Thus, we have that $ \left\langle\alpha_{1},\alpha_{2},\alpha_{3}\right\rangle \cong M_{1}  $ is one of the $ p+1 $ Sylow $ p $-subgroups of $ \mathrm{Aut}(M_{1}) $, which is the one we can, and shall, without loss of generality, work with. First, note that for $ \left\lvert\varTheta(G) \right\rvert=1 $, we have \begin{align*}
e(M_{1},M_{1},1)&=\widetilde{e}(M_{1},M_{1},1)=1  \ \text{and}\\
e(G,M_{1},1)&=\widetilde{e}(G,M_{1},1)=0 \ \text{if} \ G\neq M_{1}.
\end{align*} 
We shall deal with the cases $ \left\lvert\varTheta(G) \right\rvert=p,p^{2},p^{3} $ in the following lemmas.

It will be useful for our calculations to derive the explicit formula for $ \left(v\alpha_{1}^{a_{1}}\alpha_{2}^{ a_{2}}\alpha_{3}^{a_{3}}\right)^{r} $ for natural numbers $ r, a_{i} $ and an element $ v=\rho^{v_{1}}\sigma^{v_{2}}\tau^{v_{3}} \in M_{1} $. For this we first note that we have 
\begin{align}\label{E66}
\alpha_{1}^{a_{1}}\alpha_{2}^{a_{2}}\alpha_{3}^{a_{3}}\cdot v&=\begin{bmatrix} 1 & a_{1} & a_{3}\\ 0 & 1 & 0 \\ 0 & a_{2} & 1 \end{bmatrix}\cdot v \nonumber \\
&=\rho^{a_{1}v_{2}+\frac{1}{2}a_{2}v_{2}\left(v_{2}-1\right)+a_{3}v_{3}}v\tau^{a_{2}v_{2}}.
\end{align}
Now by using (\ref{E109}) and (\ref{E66}) we find
\begin{align}\label{E92}
\left(v\alpha_{1}^{a_{1}}\alpha_{2}^{ a_{2}}\alpha_{3}^{a_{3}}\right)^{r}&=\left(\prod_{j=0}^{r-1}\rho^{k_{j}}v\tau^{ a_{2}v_{2}j}\right)\left(\alpha_{1}^{a_{1}}\alpha_{2}^{ a_{2}}\alpha_{3}^{a_{3}}\right)^{r}\nonumber\\
&=\rho^{l_{1}}v^{r}\tau^{l_{2}a_{2}v_{2}}\left(\alpha_{1}^{a_{1}}\alpha_{2}^{ a_{2}}\alpha_{3}^{a_{3}}\right)^{r},
\end{align}
(note order of the product matters and is in increasing $ j $) with \[ k_{j}\stackrel{\mathrm{def}}{=}\left(a_{1}v_{2}j+\frac{1}{2} a_{2}a_{3}v_{2}j\left(j-1\right)+\frac{1}{2} a_{2}v_{2}\left(v_{2}-1\right)j+a_{3}v_{3}j\right), \] 
for $ j=0,...,r-1 $,
\begin{align*}
l_{1}&=l_{1}(r)\stackrel{\mathrm{def}}{=}\sum_{j=1}^{r-1}k_{j}+\frac{a_{2}v_{2}^{2}}{2}\sum_{j=1}^{r-2}j\left(j+1\right)\ \text{and}\\
l_{2}&=l_{2}(r)\stackrel{\mathrm{def}}{=}\sum_{j=1}^{r-1}j. 
\end{align*}
The second summation in $ l_{1} $ arises by moving the $ \tau^{a_{2}v_{2}j} $ terms to gather them in one place using the relation $ \tau\sigma=\rho\sigma\tau $. Note, here $ l_{1} $ and $ l_{2} $ are divisible by $ r $ for $ r>3 $ a prime number, so we find
\begin{align}\label{E110}
\left(v\alpha_{1}^{a_{1}}\alpha_{2}^{ a_{2}}\alpha_{3}^{a_{3}}\right)^{p}=1 
\end{align} for every $ v \in M_{1} $ since $ p>3 $. Note further that in (\ref{E92}), when $ a_{2}=0 $, we have
\begin{align}\label{E93}
\left(v\alpha_{1}^{a_{1}}\alpha_{3}^{a_{3}}\right)^{r}\in v^{r}\alpha_{1}^{ra_{1}}\alpha_{3}^{ra_{3}} \left\langle\rho\right\rangle,
\end{align}   
where $  \left\langle\rho\right\rangle $ is a normal subgroup of $ \mathrm{Hol}(M_{1}) $ since it is a characteristic subgroup of $ M_{1} $.

It will further be useful, when finding the non-isomorphic braces, to derive the explicit formula for a term of the form $ \alpha\left(v\alpha_{1}^{a_{1}}\alpha_{2}^{a_{2}}\alpha_{3}^{a_{3}}\right)\alpha^{-1} $ for an automorphism $ \alpha \in \mathrm{Aut}(M_{1}) $. Now if \[ \alpha=\gamma\beta \in \mathrm{Aut}(M_{1})\cong C_{p}^{2}\rtimes\mathrm{GL}_{2}(\mathbb{F}_{p}) \ \text{where} \]\[ \gamma\stackrel{\mathrm{def}}{=} \alpha_{1}^{r_{1}}\alpha_{3}^{r_{3}} \in C_{p}^{2}, \ \beta \stackrel{\mathrm{def}}{=} \begin{pmatrix} b_{1} & b_{2} \\ b_{3} & b_{4} \end{pmatrix}  \in \mathrm{GL}_{2}(\mathbb{F}_{p}),\] then, using (\ref{E95}), we have   \[\alpha\left(v\alpha_{1}^{a_{1}}\alpha_{2}^{ a_{2}}\alpha_{3}^{a_{3}}\right)\alpha^{-1}=\left(\alpha\cdot v\right)\alpha_{1}^{r_{1}}\alpha_{3}^{r_{3}}\alpha_{1}^{\left(a_{1}-a_{2} a_{3}\right)b_{4}-a_{3}b_{3}}\alpha_{3}^{-\left(a_{1}-a_{2} a_{3}\right)b_{2}+a_{3}b_{1}}\beta\alpha_{2}^{ a_{2}}\beta^{-1}\alpha_{1}^{-r_{1}}\alpha_{3}^{-r_{3}},\]
where using the section of the exact sequence in Proposition \ref{P3}, we have \[ \beta \cdot v =  \rho^{\mathrm{det}(\beta)v_{1}+\frac{1}{2}\left(b_{1}b_{3}v_{2}+b_{2}b_{4}v_{3}\right)}\left(\sigma^{b_{1}}\tau^{b_{3}}\right)^{v_{2}}\left(\sigma^{b_{2}}\tau^{b_{4}}\right)^{v_{3}}, \] 
which gives 
\begin{align}\label{E112}
\alpha\cdot &v=\rho^{\widetilde{v}_{1}}\sigma^{b_{1}v_{2}+b_{2}v_{3}}\tau^{b_{3}v_{2}+b_{4}v_{3}}, \ \text{where}\\
\widetilde{v}_{1}&\stackrel{\mathrm{def}}{=}\mathrm{det}(\beta)v_{1}+\frac{1}{2}\left(b_{3}b_{1}v_{2}^{2}+b_{4}b_{2}v_{3}^{2}\right)+b_{2}b_{3}v_{2}v_{3}+r_{1}\left(b_{1}v_{2}+b_{2}v_{3}\right)+r_{3}\left(b_{3}v_{2}+b_{4}v_{3}\right)\nonumber
\end{align}

The above implies that, when $ a_{2}=0 $, we have 
\begin{align}\label{E15}
\alpha\left(v\alpha_{1}^{a_{1}}\alpha_{3}^{a_{3}}\right)\alpha^{-1}=\left(\alpha\cdot v\right)\alpha_{1}^{a_{1}b_{4}-a_{3}b_{3}}\alpha_{3}^{a_{3}b_{1}-a_{1}b_{2}},
\end{align}
with $ \alpha\cdot v $ as given in (\ref{E112}), and when $  a_{2}\neq0 $, we can set $ b_{2}=0 $, since we want to remain within  $ \left\langle\alpha_{1},\alpha_{2},\alpha_{3}\right\rangle $, and in this case since we have  
\[\beta\alpha_{2}^{a_{2}}\beta^{-1}=\alpha_{1}^{\frac{1}{2}a_{2}b_{4}\left(b_{1}^{-1}-1\right)}\alpha_{2}^{a_{2}b_{1}^{-1}b_{4}},  \]
so (when $ b_{2}=0 $) we get 
\begin{align}\label{E16}
\alpha\left(v\alpha_{1}^{a_{1}}\alpha_{2}^{a_{2}}\alpha_{3}^{a_{3}}\right)\alpha^{-1}=\left(\alpha\cdot v\right)\alpha_{1}^{a_{1}b_{4}-a_{3}b_{3}+r_{3}a_{2}b_{1}^{-1}b_{4}+\frac{1}{2}a_{2}b_{4}\left(b_{1}^{-1}-1\right)}\alpha_{2}^{a_{2}b_{1}^{-1}b_{4}}\alpha_{3}^{a_{3}b_{1}},
\end{align}
where $ \alpha\cdot v $ can be calculated using (\ref{E112}).
\begin{lemma}\label{L16}
	For $ \left\lvert\varTheta(G) \right\rvert=p $ there are exactly $ 2(p-1) $  $ M_{1} $-skew braces of $ M_{1} $ type and two  $ C_{p}^{3} $-skew braces of $ M_{1} $ type.
\end{lemma}
\begin{proof}
	If $ G \subseteq \mathrm{Hol}(M_{1}) $ with $ \left\lvert\varTheta(G) \right\rvert=p $ is a regular subgroup, then we can assume, without loss of generality, that $ \varTheta(G) \subseteq \left\langle\alpha_{1},\alpha_{2},\alpha_{3}\right\rangle $ is a subgroup of order $ p $. We also have $ G\cap M_{1}  $ is a subgroup of order $ p^{2} $. Therefore, $ \varTheta(G) $ is one of
	\begin{align*}\label{E48}
	\left\langle\alpha_{1}^{a_{1}}\alpha_{2}^{a_{2}}\alpha_{3}^{a_{3}}\right\rangle \ \text{for} \ a_{1},a_{2},a_{3}=0,...,p-1 \ \text{with} \ (a_{1},a_{2},a_{3})\neq (0,0,0),
	\end{align*}
	(each occurring $ p-1 $ times) and $ G\cap M_{1}  $ is one of \[ \left\langle \rho,\tau \right\rangle, \left\langle \rho, \sigma\tau^{d}\right\rangle \ \text{for} \ d=0,...,p-1.\]
	Suppose we consider subgroups of the form 
	\[G=\left\langle \rho, \sigma\tau^{d},h \right\rangle \ \text{where} \ h\stackrel{\mathrm{def}}{=}\tau\alpha_{1}^{a_{1}}\alpha_{2}^{a_{2}}\alpha_{3}^{a_{3}}. \]
	Note, using (\ref{E66}), we must have \[h\left(\sigma\tau^{d}\right)h^{-1}=\tau\left(\alpha_{1}^{a_{1}}\alpha_{2}^{a_{2}}\alpha_{3}^{a_{3}}\cdot\left(\sigma\tau^{d}\right)\right)\tau^{-1}=\rho^{a_{3}d+a_{1}+1}\sigma\tau^{a_{2}+d} \in \left\langle \rho, \sigma\tau^{d}\right\rangle ,\]
	and since for a natural number $ r $  we have \[\left(\sigma\tau^{d}\right)^{r}=\rho^{\frac{1}{2}dr\left(r-1\right)}\sigma^{r}\tau^{rd},\]
	the pairing is possible, when $ a_{2}=0 $. Therefore, we consider subgroups of the form 
	\[ G=\left\langle \rho, \sigma\tau^{d},h  \right\rangle\ \text{where} \ h\stackrel{\mathrm{def}}{=}\tau\alpha_{1}^{a_{1}}\alpha_{3}^{a_{3}}.\]
	But now since the automorphism of $ M_{1} $ corresponding to $ \begin{psmallmatrix} d & -1 \\ 1-d & 1 \end{psmallmatrix} \in \mathrm{GL}_{2}(\mathbb{F}_{p}) $ maps the subgroup $ \left\langle \rho, \sigma\tau^{d}\right\rangle $ to $ \left\langle \rho, \tau\right\rangle $, we can assume every one of these skew braces is isomorphic to one containing the subgroup $ \left\langle \rho, \tau\right\rangle $.
	
	Hence, up to conjugation, we must have \[G=\left\langle \rho, \tau, g \right\rangle \ \text{where} \ g\stackrel{\mathrm{def}}{=}\sigma\alpha_{1}^{a_{1}}\alpha_{2}^{a_{2}}\alpha_{3}^{a_{3}}.\]
	Note, using (\ref{E66}), we have 
	\begin{align*}
	g\tau g^{-1}&=\sigma\left(\alpha_{1}^{a_{1}}\alpha_{2}^{a_{2}}\alpha_{3}^{a_{3}}\cdot\tau\right) \sigma^{-1}=\rho^{\left(a_{3}-1\right)}\tau \in \left\langle \rho,\tau \right\rangle \ \text{and} \\
	g\rho g^{-1}&=\sigma\left(\alpha_{1}^{a_{1}}\alpha_{2}^{a_{2}}\alpha_{3}^{a_{3}}\cdot\rho\right) \sigma^{-1}=\rho \in \left\langle \rho,\tau \right\rangle,
	\end{align*}
	so the pairing is possible. Further, it follows from (\ref{E110}) that $ g^{p}=1 $. Now, for $ r\neq 0 $, using (\ref{E66}), we have
	\begin{align}\label{E49}
	g\tau^{r}=\left(\sigma\alpha_{1}^{a_{1}}\alpha_{2}^{a_{2}}\alpha_{3}^{a_{3}}\right)\tau^{r} =\rho^{ra_{3}}\sigma\tau^{r}\alpha_{1}^{a_{1}}\alpha_{2}^{a_{2}}\alpha_{3}^{a_{3}}=\rho^{r\left(a_{3}-1\right)}\tau^{r}g,
	\end{align}
	so $ G $ is abelian if and only if $ a_{3}=1 $. Furthermore, all these subgroups are regular since they have order $ p^{3} $ and $ \left\langle \rho,\tau \right\rangle \cup \{\sigma\} \subseteq \mathrm{Orb}(1) $, i.e., since $ \lvert\mathrm{Orb}(1)\rvert>p^{2} $, their action on $ M_{1} $ is transitive. 
	
	Therefore, for $ a_{3}=1 $ we find regular subgroups isomorphic to $ C_{p}^{3} $ of the form
	\begin{align}\label{E50}
	&\left\langle \rho,\tau, \sigma\alpha_{1}^{a}\alpha_{3} \right\rangle, \left\langle \rho,\tau, \sigma\alpha_{1}^{a}\alpha_{2}^{b}\alpha_{3} \right\rangle \cong C_{p}^{3}\nonumber \\ 
	&\text{for} \ a=0,...,p-1, \ b=1,...,p-1,
	\end{align}
	and for $ a_{3}\neq 1 $, setting $ r=\left(1-a_{3}\right)^{-1} $ in (\ref{E49}), we find regular subgroups isomorphic to $ M_{1} $ of the form  
	\begin{align}\label{E51}
	&\left\langle \rho,\tau, \sigma\alpha_{1}^{b} \right\rangle,\left\langle \rho,\tau, \sigma\alpha_{1}^{a}\alpha_{2}^{b} \right\rangle, \left\langle \rho,\tau, \sigma\alpha_{1}^{a}\alpha_{3}^{c} \right\rangle, \left\langle \rho,\tau, \sigma\alpha_{1}^{a}\alpha_{2}^{b}\alpha_{3}^{c} \right\rangle \cong M_{1} \nonumber\\
	&\text{for} \ a=0,...,p-1, \ b,c=1,...,p-1\ \text{with} \ c\neq 1.
	\end{align}
	
	To find the non-isomorphic skew braces corresponding to the above regular subgroups, we let \[ \alpha=\gamma\beta \in \mathrm{Aut}(M_{1})\cong C_{p}^{2}\rtimes\mathrm{GL}_{2}(\mathbb{F}_{p}) \ \text{where} \]\[ \gamma\stackrel{\mathrm{def}}{=} \alpha_{1}^{r_{1}}\alpha_{3}^{r_{3}} \in C_{p}^{2}, \ \beta \stackrel{\mathrm{def}}{=} \begin{pmatrix} b_{1} & b_{2} \\ b_{3} & b_{4} \end{pmatrix}  \in \mathrm{GL}_{2}(\mathbb{F}_{p}),\] 
	and we work with automorphisms which fix the subgroup $ \left\langle \rho, \tau\right\rangle $, i.e., when $ b_{2}=0 $. In such case, using (\ref{E16}), we have 
	\[\alpha\left(\sigma\alpha_{1}^{a_{1}}\alpha_{2}^{a_{2}}\alpha_{3}^{a_{3}}\right)\alpha^{-1}=\left(\alpha\cdot \sigma\right)\alpha_{1}^{a_{1}b_{4}-a_{3}b_{3}+r_{3}a_{2}b_{1}^{-1}b_{4}+\frac{1}{2}a_{2}b_{4}\left(b_{1}^{-1}-1\right)}\alpha_{2}^{a_{2}b_{1}^{-1}b_{4}}\alpha_{3}^{a_{3}b_{1}},\]
	where using (\ref{E112})
	\[\alpha\cdot \sigma=\rho^{\frac{1}{2}b_{1}b_{3}-r_{1}b_{1}+r_{3}b_{3}}\sigma^{b_{1}}\tau^{b_{3}}.\]
	Now since
	\[\alpha\left(\sigma\alpha_{1}^{a}\alpha_{3}^{c}\right)\alpha^{-1}=\left(\alpha\cdot \sigma\right)\alpha_{1}^{ab_{4}-cb_{3}}\alpha_{3}^{cb_{1}}\in \sigma^{b_{1}}\alpha_{1}^{ab_{4}-cb_{3}}\alpha_{3}^{cb_{1}}\left\langle \rho, \tau\right\rangle,\]
	we have 
	\[\alpha\left(\sigma\alpha_{1}^{a}\alpha_{3}^{c}\right)^{b_{1}^{-1}}\alpha^{-1}\in \sigma\alpha_{1}^{ab_{1}^{-1}b_{4}-cb_{1}^{-1}b_{3}}\alpha_{3}^{c}\left\langle \rho, \tau\right\rangle.\]
	Thus if we conjugate the subgroup $ \left\langle \rho,\tau, \sigma\alpha_{3}^{c} \right\rangle $ with the automorphism corresponding to $ \begin{psmallmatrix} 1 & 0 \\ -ac^{-1} & 1 \end{psmallmatrix} $ we get $ \left\langle \rho,\tau, \sigma\alpha_{1}^{a}\alpha_{3}^{c} \right\rangle $, and now the subgroups $   \left\langle \rho,\tau, \sigma\alpha_{3}^{c} \right\rangle  $ for different values of $ c $ cannot be conjugate to each other.
	
	Next, working similar to above, we have 
	\[\alpha\left(\sigma\alpha_{1}^{a}\alpha_{2}^{b}\alpha_{3}^{c}\right)^{b_{1}^{-1}}\alpha^{-1}\in\sigma\alpha_{1}^{ab_{1}^{-1}b_{4}-cb_{1}^{-1}b_{3}+r_{3}bb_{1}^{-2}b_{4}+\frac{1}{2}bb_{1}^{-1}b_{4}\left(b_{1}^{-1}-1\right)\left(c+1\right)}\alpha_{2}^{bb_{1}^{-2}b_{4}}\alpha_{3}^{c}\left\langle \rho, \tau\right\rangle.\]
	Thus, if we conjugate the subgroup $ \left\langle \rho,\tau, \sigma\alpha_{2}\alpha_{3}^{c} \right\rangle $ with the automorphism corresponding to $  \begin{psmallmatrix} 1 & 0 \\ -ac^{-1} & b \end{psmallmatrix}  $, we get $ \left\langle \rho,\tau, \sigma\alpha_{1}^{a}\alpha_{2}^{b}\alpha_{3}^{c} \right\rangle $, and now again the subgroups $ \left\langle \rho,\tau, \sigma\alpha_{2}\alpha_{3}^{c} \right\rangle $ for different values of $ c $ cannot be conjugate. Finally, we note that 
	\[\alpha\left(\sigma\alpha_{1}^{a}\alpha_{2}^{b}\right)^{b_{1}^{-1}}\alpha^{-1}\in  \sigma\alpha_{1}^{ab_{1}^{-1}b_{4}+r_{3}bb_{1}^{-2}b_{4}+\frac{1}{2}bb_{1}^{-1}b_{4}\left(b_{1}^{-1}-1\right)}\alpha_{2}^{bb_{1}^{-2}b_{4}}\left\langle \rho, \tau\right\rangle,\]
	so 
	\[\alpha\left(\sigma\alpha_{1}^{a}\right)^{b_{1}^{-1}}\alpha^{-1}\in  \sigma\alpha_{1}^{ab_{1}^{-1}b_{4}}\left\langle \rho, \tau\right\rangle,\]
	which implies that conjugating the subgroup $ \left\langle \rho,\tau, \sigma\alpha_{1} \right\rangle $ with the automorphism corresponding to $  \begin{psmallmatrix} 1 & 0 \\ 0 & b \end{psmallmatrix}  $, we get $ \left\langle \rho,\tau, \sigma\alpha_{1}^{b} \right\rangle $, and conjugating the subgroup $ \left\langle \rho,\tau, \sigma\alpha_{2} \right\rangle $ with the automorphism corresponding to $  \alpha_{3}^{ab^{-1}}\begin{psmallmatrix} 1 & 0 \\ 0 & b \end{psmallmatrix}  $, we get $ \left\langle \rho,\tau, \sigma\alpha_{1}^{a}\alpha_{2}^{b} \right\rangle $.
	
	Therefore, we have non-isomorphic skew braces
	\begin{align}\label{E54}
	&\left\langle \rho,\tau, \sigma\alpha_{3} \right\rangle, \left\langle\rho, \tau, \sigma\alpha_{2}\alpha_{3}  \right\rangle   \cong C_{p}^{3}; \\ & \left\langle  \rho, \tau, \sigma\alpha_{1}  \right\rangle, \left\langle \rho,\tau, \sigma\alpha_{2} \right\rangle,\nonumber
	\left\langle \rho,\tau, \sigma\alpha_{3}^{c} \right\rangle, \left\langle\rho, \tau, \sigma\alpha_{2}\alpha_{3}^{c}  \right\rangle  \cong M_{1} \ \text{for} \ c=2,...,p-1, 
	\end{align}
	and counting them we find that there are $ 2(p-1) $ $ M_{1} $-skew braces of $ M_{1} $ type and two $ C_{p}^{3} $-skew braces of $ M_{1} $ type.	
\end{proof}
\begin{lemma}\label{L16H}
	There are \[ (p^{3}-p^{2}-1)(p+1) \] Hopf-Galois structures of $ M_{1} $ type on Galois extensions of fields with Galois group $ G\cong M_{1} $ and $ \left\lvert\varTheta(G) \right\rvert=p $, and exactly \[(p^{3}-1)(p+1)p^{2}\] Hopf-Galois structures of $ M_{1} $ type on Galois extensions of fields with Galois group $ G\cong C_{p}^{3} $ and $ \left\lvert\varTheta(G) \right\rvert=p $. 
\end{lemma}
\begin{proof}
	To find the number of Hopf-Galois structures corresponding to the skew braces in (\ref{E54}) of Lemma \ref{L16}, 
	\begin{align*}
	&\left\langle \rho,\tau, \sigma\alpha_{3} \right\rangle, \left\langle\rho, \tau, \sigma\alpha_{2}\alpha_{3}  \right\rangle   \cong C_{p}^{3};\\ &  \left\langle  \rho, \tau, \sigma\alpha_{1}  \right\rangle, \left\langle \rho,\tau, \sigma\alpha_{2} \right\rangle,\nonumber
	\left\langle \rho,\tau, \sigma\alpha_{3}^{c} \right\rangle, \left\langle\rho, \tau, \sigma\alpha_{2}\alpha_{3}^{c}  \right\rangle  \cong M_{1} \ \text{for} \ c=2,...,p-1, 
	\end{align*}
	 we need to find the automorphism groups of these skew braces. 
	 
	 We let \[ \alpha=\gamma\beta \in \mathrm{Aut}(M_{1})\ \text{where}\  \gamma\stackrel{\mathrm{def}}{=} \alpha_{1}^{r_{1}}\alpha_{3}^{r_{3}}, \ \beta \stackrel{\mathrm{def}}{=} \begin{pmatrix} b_{1} & b_{2} \\ b_{3} & b_{4} \end{pmatrix},\]   	
	and since we need $ \alpha\left( \left\langle \rho,\tau \right\rangle \right)=  \left\langle \rho,\tau \right\rangle $, we must set  $ b_{2}=0 $. Now, if $ \alpha \in \mathrm{Aut}_{\mathcal{B}r}(\left\langle \rho,\tau, \sigma\alpha_{3}^{c} \right\rangle) $, since we have
	\[\alpha\left(\sigma\alpha_{3}^{c}\right)^{b_{1}^{-1}}\alpha^{-1}\in\sigma\alpha_{1}^{-cb_{1}^{-1}b_{3}}\alpha_{3}^{c}\left\langle \rho, \tau\right\rangle,\]
	we must have $ b_{3}=0 $, thus we find
	\[\mathrm{Aut}_{\mathcal{B}r}(\left\langle \rho,\tau, \sigma\alpha_{3}^{c} \right\rangle)=\left\{ \alpha \in \mathrm{Aut}(M_{1}) \mid \alpha= \alpha_{1}^{r_{1}}\alpha_{3}^{r_{3}} \begin{psmallmatrix} b_{1} & 0 \\ 0 & b_{4} \end{psmallmatrix} \right\}.\]
	If $ \alpha \in \mathrm{Aut}_{\mathcal{B}r}(\left\langle \rho,\tau, \sigma\alpha_{2}\alpha_{3}^{c} \right\rangle) $, since we have
	\[\alpha\left(\sigma\alpha_{2}\alpha_{3}^{c}\right)^{b_{1}^{-1}}\alpha^{-1}\in\sigma\alpha_{1}^{-cb_{1}^{-1}b_{3}+r_{3}b_{1}^{-2}b_{4}+\frac{1}{2}b_{1}^{-1}b_{4}\left(b_{1}^{-1}-1\right)\left(c+1\right)}\alpha_{2}^{b_{1}^{-2}b_{4}}\alpha_{3}^{c}\left\langle \rho, \tau\right\rangle,\]
	we must have $ b_{4}=b_{1}^{2} $ and 
	\begin{align*}
	&-cb_{1}^{-1}b_{3}+r_{3}b_{1}^{-2}b_{4}+\frac{1}{2}b_{1}^{-1}b_{4}\left(b_{1}^{-1}-1\right)\left(c+1\right)\\
	&=-cb_{1}^{-1}b_{3}+r_{3}+\frac{1}{2}b_{1}\left(b_{1}^{-1}-1\right)\left(c+1\right)=0,
	\end{align*}
	so we find 
	\[\mathrm{Aut}_{\mathcal{B}r}(\left\langle \rho,\tau, \sigma\alpha_{2}\alpha_{3}^{c} \right\rangle)=\left\{ \alpha \in \mathrm{Aut}(M_{1}) \mid \alpha= \alpha_{1}^{r_{1}}\alpha_{3}^{cb_{1}^{-1}b_{3}+\frac{1}{2}\left(b_{1}-1\right)\left(c+1\right)} \begin{psmallmatrix} b_{1} & 0 \\ b_{3} & b_{1}^{2} \end{psmallmatrix} \right\}.\]
	If $ \alpha \in \mathrm{Aut}_{\mathcal{B}r}(\left\langle  \rho, \tau, \sigma\alpha_{1}  \right\rangle) $, since we have
	\[\alpha\left(\sigma\alpha_{1}\right)^{b_{1}^{-1}}\alpha^{-1}\in\sigma\alpha_{1}^{b_{1}^{-1}b_{4}}\left\langle \rho, \tau\right\rangle,\]
	we must have $ b_{1}=b_{4} $, and we find 
	\[\mathrm{Aut}_{\mathcal{B}r}(\left\langle  \rho, \tau, \sigma\alpha_{1}  \right\rangle)=\left\{ \alpha \in \mathrm{Aut}(M_{1}) \mid \alpha= \alpha_{1}^{r_{1}}\alpha_{3}^{r_{3}} \begin{psmallmatrix} b_{1} & 0 \\ b_{3} & b_{1} \end{psmallmatrix} \right\}.\]
	Finally, if $ \alpha \in \mathrm{Aut}_{\mathcal{B}r}(\left\langle \rho,\tau, \sigma\alpha_{2} \right\rangle) $, since we have
	\[\alpha\left(\sigma\alpha_{2}\right)^{b_{1}^{-1}}\alpha^{-1}\in\sigma\alpha_{1}^{r_{3}b_{1}^{-2}b_{4}+\frac{1}{2}b_{1}^{-1}b_{4}\left(b_{1}^{-1}-1\right)}\alpha_{2}^{b_{1}^{-2}b_{4}}\left\langle \rho, \tau\right\rangle,\]
	we must have $ b_{4}=b_{1}^{2} $ and $ r_{3}=\frac{1}{2}(b_{1}-1) $, we find 
	\[\mathrm{Aut}_{\mathcal{B}r}(\left\langle \rho,\tau, \sigma\alpha_{2} \right\rangle)=\left\{ \alpha \in \mathrm{Aut}(M_{1}) \mid \alpha= \alpha_{1}^{r_{1}}\alpha_{3}^{\frac{1}{2}\left(b_{1}-1\right)} \begin{psmallmatrix} b_{1} & 0 \\ b_{3} & b_{1}^{2} \end{psmallmatrix} \right\}.\]
	
	Therefore, we have 
	\begin{align*}
	&e(M_{1},M_{1},p)= \sum_{(M_{1})_{M_{1}}(p)}\dfrac{\left\lvert\mathrm{Aut}(M_{1})\right\rvert}{\left\lvert\mathrm{Aut}_{\mathcal{B}r}((M_{1})_{M_{1}})\right\rvert}=\\
	&\dfrac{\left\lvert\mathrm{Aut}(M_{1})\right\rvert}{\left\lvert\mathrm{Aut}_{\mathcal{B}r}(\left\langle  \rho, \tau, \sigma\alpha_{1}  \right\rangle)\right\rvert}+\dfrac{\left\lvert\mathrm{Aut}(M_{1})\right\rvert}{\left\lvert \mathrm{Aut}_{\mathcal{B}r}(\left\langle \rho,\tau, \sigma\alpha_{2} \right\rangle)\right\rvert }+\sum_{c=2}^{p-1}\dfrac{\left\lvert\mathrm{Aut}(M_{1})\right\rvert}{\left\lvert\mathrm{Aut}_{\mathcal{B}r}(\left\langle \rho,\tau, \sigma\alpha_{3}^{c} \right\rangle)\right\rvert}+\dfrac{\left\lvert\mathrm{Aut}(M_{1})\right\rvert}{ \left\lvert\mathrm{Aut}_{\mathcal{B}r}(\left\langle \rho,\tau, \sigma\alpha_{2}\alpha_{3}^{c} \right\rangle)\right\rvert}\\
	&=(p^{2}-1)(p-1)p^{3}\left(\dfrac{1}{(p-1)p^{3}}+\dfrac{1}{(p-1)p^{2}}+\sum_{c=2}^{p-1}\dfrac{1}{(p-1)^{2}p^{2}}+\dfrac{1}{(p-1)p^{2}}\right)\\
	&=(p^{3}-p^{2}-1)(p+1),
	\end{align*}
	and similarly 
	\begin{align*}
	&e(C_{p}^{3},M_{1},p)= \sum_{(C_{p}^{3})_{M_{1}}(p)}\dfrac{\left\lvert\mathrm{Aut}(C_{p}^{3})\right\rvert}{\left\lvert\mathrm{Aut}_{\mathcal{B}r}((C_{p}^{3})_{M_{1}})\right\rvert}=\\
	&\dfrac{\left\lvert\mathrm{Aut}(C_{p}^{3})\right\rvert}{\left\lvert\mathrm{Aut}_{\mathcal{B}r}(\left\langle \rho,\tau, \sigma\alpha_{3} \right\rangle)\right\rvert}+\dfrac{\left\lvert\mathrm{Aut}(C_{p}^{3})\right\rvert}{ \left\lvert\mathrm{Aut}_{\mathcal{B}r}(\left\langle \rho,\tau, \sigma\alpha_{2}\alpha_{3} \right\rangle)\right\rvert}\\
	&=(p^{3}-1)(p^{3}-p)(p^{3}-p^{2})\left(\dfrac{1}{(p-1)^{2}p^{2}}+\dfrac{1}{(p-1)p^{2}}\right)=(p^{3}-1)(p+1)p^{2}.
	\end{align*}
\end{proof}
\begin{lemma}\label{L17}
	For $ \left\lvert\varTheta(G) \right\rvert=p^{2} $ there are exactly 
	$ (2p-3)p $ $ M_{1} $-skew braces of $ M_{1} $ type and $ 2p-1 $ $ C_{p}^{3} $-skew braces of $ M_{1} $ type.
\end{lemma}
\begin{proof}
	If $ G \subseteq \mathrm{Hol}(M_{1}) $ with $ \left\lvert\varTheta(G) \right\rvert=p^{2} $ is a regular subgroup, then we can assume, without loss of generality, that we have $ \varTheta(G) \subseteq \left\langle\alpha_{1},\alpha_{2},\alpha_{3}\right\rangle $ a subgroup of order $ p^{2} $. We also have $ G\cap M_{1}  $ a subgroup of order $ p $. Therefore, $ \varTheta(G) $ is one of \[\left\langle\alpha_{1},\alpha_{3}\right\rangle,\left\langle\alpha_{1},\alpha_{2}\alpha_{3}^{a}\right\rangle \ \text{for} \  a=0,...,p-1, \]
	and $ G\cap M_{1} $ is of the form \[  \left\langle\rho^{b}\sigma^{c}\tau^{d}\right\rangle \ \text{for} \ b,c,d=0,...,p-1 \ \text{with} \ \left(b,c,d\right)\neq\left(0,0,0\right), \]
	each occurring $ p-1 $ times. We shall consider all subgroups of order $ p $ in $ M_{1} $ and all ways of pairing them with a subgroup of order $ p^{2} $ of $ \left\langle\alpha_{1},\alpha_{2},\alpha_{3}\right\rangle $.
	
	Let us consider a subgroup of the form 
	\[G=\left\langle u ,v\alpha_{1}, w\alpha_{2}^{a_{2}}\alpha_{3}^{a_{3}}\right\rangle  \ \text{for} \ \left(a_{2},a_{3}\right)\neq \left(0,0\right), \ u,v,w\neq 1. \]
	Suppose $ u=\rho^{u_{1}}\sigma^{u_{2}}\tau^{u_{3}} $, $ v=\rho^{v_{1}}\sigma^{v_{2}}\tau^{v_{3}} $, and $ w=\rho^{w_{1}}\sigma^{w_{2}}\tau^{w_{3}} $. Then, we need the following.
	\begin{align}\label{E17}
	\left(v\alpha_{1}\right)u\left(v\alpha_{1}\right)^{-1}=v\left(\alpha_{1}\cdot u \right)v^{-1}u^{-1}=\rho^{u_{2}+u_{2}v_{3}-u_{3}v_{2}} \in \left\langle u\right\rangle,
	\end{align} 
	\begin{align}\label{E18}
	\left(w\alpha_{2}^{a_{2}}\alpha_{3}^{a_{3}}\right)u\left(w\alpha_{2}^{a_{2}}\alpha_{3}^{a_{3}}\right)^{-1}&=w\left(\alpha_{2}^{a_{2}}\alpha_{3}^{a_{3}}\cdot u \right)w^{-1}u^{-1}= \nonumber \\
	&\rho^{\frac{1}{2}a_{2}u_{2}\left(u_{2}-1\right)+a_{3}u_{3}+u_{2}w_{3}-u_{3}w_{2}-a_{2}u_{2}w_{2}-a_{2}u_{2}^{2}}\tau^{a_{2}u_{2}} \in \left\langle u\right\rangle,
	\end{align}
	\begin{align}\label{E19}
	&\left(v\alpha_{1}\right)\left(w \alpha_{2}^{a_{2}}\alpha_{3}^{a_{3}}\right)\left(\left(w \alpha_{2}^{a_{2}}\alpha_{3}^{a_{3}}\right)\left(v\alpha_{1}\right)\right)^{-1}=\nonumber\\
	&\left(\rho^{w_{2}}vw\alpha_{1}\alpha_{2}^{a_{2}}\alpha_{3}^{a_{3}}\right)\left(\rho^{\frac{1}{2}a_{2}v_{2}\left(v_{2}-1\right)+a_{3}v_{1}-a_{2}v_{2}^{2}+v_{2}w_{1}-v_{1}w_{2}}\tau^{a_{2}v_{2}}vw\alpha_{1}\alpha_{2}^{a_{2}}\alpha_{3}^{a_{3}}\right)^{-1}\nonumber\\
	&=\rho^{w_{2}-\frac{1}{2}a_{2}v_{2}\left(v_{2}-1\right)-a_{3}v_{1}+a_{2}v_{2}^{2}-v_{2}w_{1}+v_{1}w_{2}}\tau^{-a_{2}v_{2}}\in \left\langle u\right\rangle.
	\end{align}
	Now assume $ u_{3}=1 $. Then, multiplying $ v\alpha_{1} $ and $ w \alpha_{2}^{a_{2}}\alpha_{3}^{a_{3}} $ by suitable powers of $ u $ if necessary, we can further assume $ v_{3}=w_{3}=0 $. Now (\ref{E17}) implies that $ u_{2}=v_{2} $ and (\ref{E18}) implies that we need \[\rho^{\frac{1}{2}a_{2}u_{2}\left(u_{2}-1\right)+a_{3}-w_{2}-a_{2}u_{2}w_{2}-a_{2}u_{2}^{2}}\tau^{a_{2}u_{2}} \in \left\langle \rho^{u_{1}}\sigma^{u_{2}}\tau\right\rangle,\]
	so $ u_{2}=v_{2}=0 $ and $ a_{3}=w_{2} $. In such case (\ref{E19}) implies that we need
	\[\rho^{w_{2}}\in \left\langle \rho^{u_{1}}\sigma^{u_{2}}\tau\right\rangle,\]
	so $ w_{2}=0 $, which implies that $ G $ cannot be regular. Thus, we cannot have any pairing with subgroups of the form $ \left\langle\rho^{b}\sigma^{c}\tau\right\rangle $. Similarly, if $ u_{2}=1 $, then we can assume $ v_{2}=w_{2}=0 $. Now (\ref{E17}) gives $ v_{3}=-1 $, also (\ref{E18}) gives $ a_{2}=0 $, and (\ref{E19}) gives $ a_{3}=0 $ which is not possible. Thus, the only possibility for $ u $ is $ u=\rho $ and then (\ref{E19}) implies that we also need $ a_{2}v_{2}=0 $.
	
	Therefore, we may only consider subgroups of the form 
	\[G=\left\langle \rho ,v\alpha_{1}, w \alpha_{2}^{a_{2}}\alpha_{3}^{a_{3}}\right\rangle \ \text{with} \ a_{2}v_{2}=v_{1}=w_{1}=0. \]
	There are two main cases to consider. 
	
	\textbf{Case I:} Let us consider 
	\[G=\left\langle\rho,u\alpha_{1},v\alpha_{3} \right\rangle.\]
	Then $ \left(u\alpha_{1}\right)\rho=\rho\left(u\alpha_{1}\right) $ and $ \left(v\alpha_{3}\right)\rho=\rho\left(v\alpha_{3}\right) $, also we have  
	\begin{align}\label{E58}
	\left(u\alpha_{1}\right)\left(v\alpha_{3}\right)&=\rho^{v_{2}}uv\alpha_{1}\alpha_{3} \ \text{and} \nonumber \\
	\left(v\alpha_{3}\right)\left(u\alpha_{1}\right)&=\rho^{u_{3}}vu\alpha_{1}\alpha_{3}=\rho^{u_{3}+u_{2}v_{3}-u_{3}v_{2}}uv\alpha_{1}\alpha_{3},
	\end{align}
	so $ G $ has order $ p^{3} $ and is abelian if and only if $ v_{2}\equiv u_{3}+u_{2}v_{3}-u_{3}v_{2} \ \mathrm{mod}\ p $; furthermore, for $ G $ to be regular we need $ u_{2}v_{3}-u_{3}v_{2}\not\equiv 0 \ \mathrm{mod}\ p $. 
	
	Therefore, for $ u_{2}v_{3}-u_{3}v_{2}\not\equiv 0 \ \mathrm{mod}\ p  $ we have regular subgroups isomorphic to $ C_{p}^{3} $ of the form 
	\begin{align}\label{E55} 
	&\left\langle \rho, u\alpha_{1}, v\alpha_{3} \right\rangle \cong C_{p}^{3} \\ 
	&\text{for} \ A=\begin{pmatrix} u_{2} & v_{2} \\ u_{3} & v_{3}  \end{pmatrix}\in \mathrm{GL}_{2}(\mathbb{F}_{p}) \ \text{with} \ v_{2}=u_{3}+\mathrm{det}(A).\nonumber
	\end{align}
	For $ v_{2}- u_{3}-u_{2}v_{3}+u_{3}v_{2}\not\equiv0 \ \mathrm{mod}\ p $, we find regular subgroups isomorphic to $ M_{1} $ of the form
	\begin{align}\label{E56}
	&\left\langle \rho, u\alpha_{1}, v\alpha_{3} \right\rangle \cong M_{1}  \\ 
	&\text{for} \ A=\begin{pmatrix} u_{2} & v_{2} \\ u_{3} & v_{3}  \end{pmatrix}\in \mathrm{GL}_{2}(\mathbb{F}_{p}) \ \text{with} \ v_{2}-u_{3}-\mathrm{det}(A)\not\equiv 0 \ \mathrm{mod}\ p\nonumber. 
	\end{align}
	
	To find the non-isomorphic skew braces corresponding to the above regular subgroups, we let $ \beta_{0}\stackrel{\mathrm{def}}{=}\begin{psmallmatrix} u_{2} & v_{2} \\ u_{3} & v_{3}  \end{psmallmatrix} $ and note that considering (\ref{E112}) and (\ref{E16}), it suffices to work with an automorphism corresponding to $ \beta \stackrel{\mathrm{def}}{=} \begin{psmallmatrix} b_{1} & b_{2} \\ b_{3} & b_{4} \end{psmallmatrix}  \in \mathrm{GL}_{2}(\mathbb{F}_{p}) $ with $ b\stackrel{\mathrm{def}}{=} \mathrm{det}(\beta)^{-1} $, and we find
	\begin{align*}
	\beta\left(u\alpha_{1}\right)^{b_{1}b}\left(v\alpha_{3}\right)^{b_{2}b}\beta^{-1}&=\rho^{\kappa_{1}}\left(b\beta\beta_{0}\beta^{T}\right)\cdot \sigma \alpha_{1},\\
	\beta\left(u\alpha_{1}\right)^{b_{3}b}\left(v\alpha_{3}\right)^{b_{4}b}\beta^{-1}&=\rho^{\kappa_{2}}\left(b\beta\beta_{0}\beta^{T}\right)\cdot \tau \alpha_{3} 
	\end{align*}
	for some $ \kappa_{1},\kappa_{2} $, where superscript $ T $ denotes the transpose of a matrix. 
	
	Now if $ u_{2}\neq 0 $, then 
	\[u_{2}^{-1}\begin{pmatrix} 1 & 0 \\  - u_{3} & u_{2} \end{pmatrix}\begin{pmatrix} u_{2} & v_{2} \\  u_{3} & v_{3} \end{pmatrix}\begin{pmatrix} 1 & 0 \\  - u_{3} & u_{2} \end{pmatrix}^{T} =\begin{pmatrix} 1 & v_{2}-u_{3} \\  0 & \mathrm{det}(\beta_{0}) \end{pmatrix}; \]
	if $ v_{3}\neq 0 $, then 
	\[v_{3}^{-1}\begin{pmatrix} 0 & 1 \\  -v_{3} & v_{2} \end{pmatrix} \begin{pmatrix} u_{2} & v_{2} \\  u_{3} & v_{3} \end{pmatrix}\begin{pmatrix} 0 & 1 \\  -v_{3} & v_{2} \end{pmatrix}^{T} =\begin{pmatrix} 1 & v_{2}-u_{3} \\  0 & \mathrm{det}(\beta_{0}) \end{pmatrix}; \]
	if $ u_{2}=v_{3}=0 $ and $ u_{3}\neq -v_{2} $, then 
	\[\left(u_{3}+v_{2}\right)^{-1}\begin{pmatrix} 1 & 1 \\   -u_{3} & v_{2} \end{pmatrix} \begin{pmatrix} 0 & v_{2} \\  u_{3} & 0 \end{pmatrix}\begin{pmatrix} 1 & 1 \\   -u_{3} & v_{2} \end{pmatrix}^{T} =\begin{pmatrix} 1 & v_{2}-u_{3} \\  0 & \mathrm{det}(\beta_{0}) \end{pmatrix}, \]
	and finally if $ u_{2}=v_{3}=0 $ and $ u_{3}= -v_{2} $, then 
	\[bI\beta_{0}I^{T}=\beta_{0}. \]
	Thus every one of our regular subgroups above is conjugate to one of the form 
	\[\left\langle \rho, \sigma\alpha_{1}, \sigma^{t_{2}}\tau^{t_{3}}\alpha_{3} \right\rangle, \left\langle \rho, \tau^{-t_{4}}\alpha_{1}, \sigma^{t_{4}}\alpha_{3} \right\rangle  \ \text{for some} \ t_{2},t_{3},t_{4},   \]
	and these for different values of $ t_{2}, t_{3} $, and $ t_{4} $ are not conjugate to each other. 
	
	Therefore, we find non-isomorphic skew braces
	\begin{align}\label{E57}
	&\left\langle \rho, \sigma\alpha_{1}, \sigma^{u_{2}}\tau^{u_{2}}\alpha_{3} \right\rangle,\left\langle \rho, \tau^{-2}\alpha_{1}, \sigma^{2}\alpha_{3} \right\rangle \cong C_{p}^{3},  \\
	&\left\langle \rho, \sigma\alpha_{1}, \sigma^{u_{3}}\tau^{u_{4}}\alpha_{3} \right\rangle,\left\langle \rho, \tau^{-u_{5}}\alpha_{1}, \sigma^{u_{5}}\alpha_{3} \right\rangle \cong M_{1} \nonumber\\
	&\text{for} \ u_{4}=0,...,p-1,\ u_{2},u_{3},u_{5}=1,...,p-1 \ \text{with} \ u_{5}\neq 2,\ u_{3}-u_{4}\not\equiv 0 \ \mathrm{mod}\ p.\nonumber
	\end{align}
	
	\textbf{Case II:} Next, we consider subgroups of the form \[G=\left\langle\rho,x\alpha_{1},y\alpha_{2}\alpha_{3}^{a} \right\rangle \ \text{with} \ x_{2}=0.\]
	Note, we have 
	\begin{align}\label{E59}
	\left(x\alpha_{1}\right)\left(y\alpha_{2}\alpha_{3}^{a}\right)&=\rho^{y_{2}}xy\alpha_{1}\alpha_{2}\alpha_{3}^{a} \ \text{and} \nonumber \\
	\left(y\alpha_{2}\alpha_{3}^{a}\right)\left(x\alpha_{1}\right)&=\rho^{ax_{3}-x_{3}y_{2}}xy\alpha_{1}\alpha_{2}\alpha_{3}^{a}, 
	\end{align}
	so $ G $ is abelian if and only if $ y_{2}\equiv ax_{3}-x_{3}y_{2}\ \mathrm{mod}\ p $; furthermore, we need $ x_{3},y_{2}\neq 0 $ for $ G $ to be regular. 
	
	Therefore, for $ y_{2}\equiv ax_{3}-x_{3}y_{2}\ \mathrm{mod}\ p $ we find regular subgroups isomorphic to $ C_{p}^{3} $ of the form 
	\begin{align}\label{E60}
	&\left\langle\rho,\tau^{x_{3}}\alpha_{1},\sigma^{y_{2}}\tau^{y_{3}}\alpha_{2}\alpha_{3}^{\left(1+x_{3}\right)y_{2}x_{3}^{-1}} \right\rangle \cong C_{p}^{3}\\
	&\text{for} \ y_{3}=0,...,p-1, \ y_{2},x_{3}=1,...,p-1, \nonumber
	\end{align}
	and for $ ax_{3}\not\equiv y_{2}+x_{3}y_{2} \ \mathrm{mod}\ p $, we find regular subgroups isomorphic to  $ M_{1} $ of the form
	\begin{align}\label{E61}
	&\left\langle\rho,\tau^{x_{3}}\alpha_{1},y\alpha_{2}\alpha_{3}^{a} \right\rangle \cong M_{1} \\
	&\text{for}\ a,y_{3}=0,...,p-1,\ x_{3},y_{2}=1,...,p-1 \ \text{with} \ ax_{3}- y_{2}-x_{3}y_{2}\not\equiv 0 \ \mathrm{mod}\ p.\nonumber
	\end{align}
	
	To find the non-isomorphic skew braces corresponding to the above regular subgroups, it suffices to work with automorphisms corresponding to elements of the form $ \beta \stackrel{\mathrm{def}}{=} \begin{psmallmatrix} b_{1} & 0 \\ b_{3} & b_{4} \end{psmallmatrix}  \in \mathrm{GL}_{2}(\mathbb{F}_{p}) $. Then, using (\ref{E112}) and (\ref{E16}), we have
	\begin{align*}
	&\left(\alpha_{3}^{r_{3}}\beta\right)\left(\tau^{x_{3}}\alpha_{1}\right)^{b_{4}^{-1}}\left(\alpha_{3}^{r_{3}}\beta\right)^{-1}=\rho^{\kappa_{1}}\tau^{x_{3}}\alpha_{1} \ \text{and} \\
	&\left(\alpha_{3}^{r_{3}}\beta\right)\left(\tau^{x_{3}}\alpha_{1}\right)^{ab_{1}b_{3}b_{4}^{-2}-r_{3}b_{4}^{-1}-\frac{1}{2}b_{4}^{-1}\left(1-b_{1}\right)-\frac{1}{2}ab_{1}b_{4}^{-1}\left(b_{1}b_{4}^{-1}-1\right)}\left(y\alpha_{2}\alpha_{3}^{a}\right)^{b_{1}b_{4}^{-1}}\left(\alpha_{3}^{r_{3}}\beta\right)^{-1}\\
	&=\rho^{\kappa_{2}}\sigma^{y_{2}b_{1}^{2}b_{4}^{-1}}\tau^{\left(ab_{1}b_{3}b_{4}^{-2}-r_{3}b_{4}^{-1}-\frac{1}{2}b_{4}^{-1}\left(1-b_{1}\right)-\frac{1}{2}ab_{1}b_{4}^{-1}\left(b_{1}b_{4}^{-1}-1\right)\right)x_{3}+b_{1}y_{3}+\frac{1}{2}b_{1}\left(b_{1}b_{4}^{-1}-1\right)y_{2}}\alpha_{2}\alpha_{3}^{ab_{1}^{2}b_{4}^{-1}},
	\end{align*}
	for some $ \kappa_{1},\kappa_{2} $, and $ r_{3} $. Now conjugating the subgroup $  \left\langle\rho,\tau^{x_{3}}\alpha_{1},y\alpha_{2}\alpha_{3}^{a} \right\rangle $ with the automorphism corresponding to $ \alpha_{3}^{\frac{1}{2}(y_{2}^{-1}-1)-y_{2}x_{3}^{-1}}\begin{psmallmatrix} y_{2}^{-1} & 0 \\ 0 & y_{2}^{-1} \end{psmallmatrix} $ we get $  \left\langle\rho,\tau^{x_{3}}\alpha_{1},\sigma\alpha_{2}\alpha_{3}^{ay_{2}^{-1}} \right\rangle  $, and these subgroups for different values of $ a $ and $ x_{3} $ and $ y_{2} $ are not conjugate to each other. 
	
	Therefore, we find non-isomorphic skew braces
	\begin{align}\label{E62}
	&\left\langle\rho,\tau^{x_{3}}\alpha_{1},\sigma\alpha_{2}\alpha_{3}^{(1+x_{3})x_{3}^{-1}} \right\rangle\cong C_{p}^{3},\ \left\langle\rho,\tau^{x_{3}}\alpha_{1},\sigma\alpha_{2}\alpha_{3}^{a} \right\rangle\cong M_{1}  \\
	& \text{for} \ a=0,...,p-1, \ x_{3}=1,...,p-1 \ \text{with} \ a- (1+x_{3})x_{3}^{-1}\not\equiv 0 \ \mathrm{mod}\ p. \nonumber
	\end{align}
	Thus, the corresponding non-isomorphic skew braces, combining (\ref{E57}) and (\ref{E62}), are \[\left\langle \rho, \sigma\alpha_{1}, \sigma^{u_{3}}\tau^{u_{4}}\alpha_{3} \right\rangle,\left\langle \rho, \tau^{-u_{5}}\alpha_{1}, \sigma^{u_{5}}\alpha_{3} \right\rangle,\left\langle \rho, \tau^{x_{3}}\alpha_{1}, \sigma\alpha_{2}\alpha_{3}^{a} \right\rangle \cong M_{1}, \]
	\[\left\langle \rho, \sigma\alpha_{1}, \sigma^{u_{2}}\tau^{u_{2}}\alpha_{3} \right\rangle,\left\langle \rho, \tau^{-2}\alpha_{1}, \sigma^{2}\alpha_{3} \right\rangle,\left\langle \rho, \tau^{x_{3}}\alpha_{1}, \sigma\alpha_{2}\alpha_{3}^{\left(1+x_{3}\right)x_{3}^{-1}} \right\rangle \cong C_{p}^{3} \ \text{for} \]
	\[a,u_{3}=0,...,p-1, \  u_{2},u_{4},u_{5},x_{3},=1,...,p-1 \]
	\[\text{with}  \ u_{5}\neq2, \ u_{3}-u_{4},\ ax_{3}-\left(1+x_{3}\right)\not\equiv 0 \ \mathrm{mod}\ p.  \]
	Therefore, there are
	\[(p-1)p-(p-1)+(p-2)+(p-1)p-(p-1)=(2p-3)p\]
	$ M_{1} $-skew braces of $ M_{1} $ type and \[(p-1)+1+(p-1)=2p-1\] $ C_{p}^{3} $-skew braces of $ M_{1} $ type.
\end{proof}
\begin{lemma}\label{L17H}
	There are \[ (p^{4}-p^{3}-2p^{2}+2p+1)p \] Hopf-Galois structures of $ M_{1} $ type on Galois extensions of fields with Galois group $ G\cong M_{1} $ and $ \left\lvert\varTheta(G) \right\rvert=p^{2} $, and exactly \[(p^{3}-1)(p^{2}-2)p^{2}\] Hopf-Galois structures of $ M_{1} $ type on Galois extensions of fields with Galois group $ G\cong C_{p}^{3} $ and $ \left\lvert\varTheta(G) \right\rvert=p^{2} $.
\end{lemma}
\begin{proof}
	To find the number of  Hopf-Galois structures corresponding to the skew braces of Lemma \ref{L17}, we need to find the automorphism groups of the skew braces \[\left\langle \rho, \sigma\alpha_{1}, \sigma^{u_{3}}\tau^{u_{4}}\alpha_{3} \right\rangle,\left\langle \rho, \tau^{-u_{5}}\alpha_{1}, \sigma^{u_{5}}\alpha_{3} \right\rangle,\left\langle \rho, \tau^{x_{3}}\alpha_{1}, \sigma\alpha_{2}\alpha_{3}^{a} \right\rangle \cong M_{1}, \]
	\[\left\langle \rho, \sigma\alpha_{1}, \sigma^{u_{2}}\tau^{u_{2}}\alpha_{3} \right\rangle,\left\langle \rho, \tau^{-2}\alpha_{1}, \sigma^{2}\alpha_{3} \right\rangle,\left\langle \rho, \tau^{x_{3}}\alpha_{1}, \sigma\alpha_{2}\alpha_{3}^{\left(1+x_{3}\right)x_{3}^{-1}} \right\rangle \cong C_{p}^{3} \ \text{for} \]
	\[a,u_{3}=0,...,p-1, \  u_{2},u_{4},u_{5},x_{3},=1,...,p-1 \]
	\[\text{with}  \ u_{5}\neq2, \ u_{3}-u_{4},\ ax_{3}-\left(1+x_{3}\right)\not\equiv 0 \ \mathrm{mod}\ p.  \] We let \[ \alpha=\gamma\beta \in \mathrm{Aut}(M_{1})\ \text{where}\  \gamma\stackrel{\mathrm{def}}{=} \alpha_{1}^{r_{1}}\alpha_{3}^{r_{3}}, \ \beta \stackrel{\mathrm{def}}{=} \begin{pmatrix} b_{1} & b_{2} \\ b_{3} & b_{4} \end{pmatrix},\]
	and set $ b \stackrel{\mathrm{def}}{=} \mathrm{det}(\beta)^{-1} $. 
	
	\textbf{For skew braces of Case I of Lemma \ref{L17}}: If $ \alpha \in \mathrm{Aut}_{\mathcal{B}r}(\left\langle \rho, \sigma\alpha_{1}, \sigma^{u_{2}}\tau^{u_{3}}\alpha_{3} \right\rangle) $, since we have
	\begin{align*}
	\alpha\left(\sigma\alpha_{1}\right)^{b_{1}b}\left(\sigma^{u_{2}}\tau^{u_{3}}\alpha_{3}\right)^{b_{2}b}\alpha^{-1}&=\rho^{\kappa_{1}}\left(b\beta\begin{psmallmatrix} 1 & u_{2} \\ 0 & u_{2} \end{psmallmatrix}\beta^{T}\right)\cdot \sigma \alpha_{1},\\
	\alpha\left(\sigma\alpha_{1}\right)^{b_{3}b}\left(\sigma^{u_{2}}\tau^{u_{3}}\alpha_{3}\right)^{b_{4}b}\alpha^{-1}&=\rho^{\kappa_{2}}\left(b\beta\begin{psmallmatrix} 1 & u_{2} \\ 0 & u_{2} \end{psmallmatrix}\beta^{T}\right)\cdot \tau \alpha_{3}, 
	\end{align*}
	we must have \[ b\beta\begin{pmatrix} 1 & u_{2} \\ 0 & u_{3} \end{pmatrix}\beta^{T}= b\begin{pmatrix} b_{1}^{2}+b_{2}(b_{1}u_{2}+b_{2}u_{3}) & b_{1}(b_{3}+b_{4}u_{2})+b_{2}b_{4}u_{3} \\ b_{1}b_{3}+b_{2}(b_{3}u_{2}+b_{4}u_{3}) & b_{3}^{2}+b_{4}(b_{3}u_{2}+b_{4}u_{3}) \end{pmatrix} =\begin{pmatrix} 1 & u_{2} \\ 0 & u_{3} \end{pmatrix} .\]
	Thus we need   
	\begin{align*}
	b_{1}^{2}+b_{2}(b_{1}u_{2}+b_{2}u_{3})&=b_{1}b_{4}-b_{2}b_{3}\\
	b_{1}b_{3}+b_{2}(b_{3}u_{2}+b_{4}u_{3})&=0\\
	b_{3}^{2}+b_{4}(b_{3}u_{2}+b_{4}u_{3})&=(b_{1}b_{4}-b_{2}b_{3})u_{3}.
	\end{align*}
	The second and third equations give
	\begin{align*}
	b_{1}b_{3}b_{4}+b_{2}b_{4}(b_{3}u_{2}+b_{4}u_{3})&=0\\
	b_{2}b_{3}^{2}+b_{2}b_{4}(b_{3}u_{2}+b_{4}u_{3})&=b_{2}(b_{1}b_{4}-b_{2}b_{3})u_{3},
	\end{align*}
	so we must have \[-b_{1}b_{3}b_{4}+b_{2}b_{3}^{2}=b_{2}(b_{1}b_{4}-b_{2}b_{3})u_{3}, \]
	which implies that we must set $ b_{3}=-b_{2}u_{3} $ and $ b_{4}=b_{1}+b_{2}u_{2} $ which satisfies all three equations. Thus we must have 
	\[\mathrm{Aut}_{\mathcal{B}r}(\left\langle \rho, \sigma\alpha_{1}, \sigma^{u_{2}}\tau^{u_{3}}\alpha_{3} \right\rangle)=\left\{ \alpha \in \mathrm{Aut}(M_{1}) \mid \alpha= \alpha_{1}^{r_{1}}\alpha_{3}^{r_{3}} \begin{psmallmatrix} b_{1} & b_{2} \\ -b_{2}u_{3} & b_{1}+b_{2}u_{2} \end{psmallmatrix} \right\},\]
	where we need $ b_{1}^{2}+b_{1}b_{2}u_{2}+b_{2}^{2}u_{3}\neq 0 $, i.e., \[ \left(b_{1}u_{2}+2b_{2}u_{3}\right)^{2}\neq b_{1}^{2}\left(u_{2}^{2}-4u_{3}\right). \]
	We now need to consider three cases for $ u_{2}^{2}-4u_{3}=0 $ and when $ u_{2}^{2}-4u_{3} $ is a square modulo $ p $ or not. We find 
	\begin{align*}
	\left\lvert\mathrm{Aut}_{\mathcal{B}r}\left(\left\langle \rho, \sigma\alpha_{1}, \sigma^{u_{2}}\tau^{u_{2}^{2}/4}\alpha_{3} \right\rangle\right)\right\rvert&=(p-1)p^{3} \ \text{for} \ u_{2}\neq 0,\\
	\left\lvert\mathrm{Aut}_{\mathcal{B}r}\left(\left\langle \rho, \sigma\alpha_{1}, \sigma^{u_{2}}\tau^{u_{3}}\alpha_{3} \right\rangle\right)\right\rvert&=(p-1)^{2}p^{2} \ \text{if} \ u_{3}\neq 0\ \text{and} \ u_{2}^{2}-4u_{3}\neq 0 \ \text{is a square}, \\
	\left\lvert\mathrm{Aut}_{\mathcal{B}r}\left(\left\langle \rho, \sigma\alpha_{1}, \sigma^{u_{2}}\tau^{u_{3}}\alpha_{3} \right\rangle\right)\right\rvert&=(p^{2}-1)p^{2} \ \text{if} \ u_{3}\neq 0 \ \text{and}\ u_{2}^{2}-4u_{3}\neq 0 \ \text{is not a square}.
	\end{align*}
	We also have \[\mathrm{Aut}_{\mathcal{B}r}(\left\langle \rho, \tau^{-v_{2}}\alpha_{1}, \sigma^{v_{2}}\alpha_{3} \right\rangle)=\left\{ \alpha \in \mathrm{Aut}(M_{1}) \mid \alpha= \alpha_{1}^{r_{1}}\alpha_{3}^{r_{3}} \begin{psmallmatrix} b_{1} & b_{2} \\ b_{3} & b_{4} \end{psmallmatrix} \right\}.\]
	
	\textbf{For skew braces of Case II of Lemma \ref{L17}}: If $ \alpha \in \mathrm{Aut}_{\mathcal{B}r}(\left\langle\rho,\tau^{x_{3}}\alpha_{1},\sigma\alpha_{2}\alpha_{3}^{a} \right\rangle) $, we need to set $ b_{2}=0 $, now since we have
	\begin{align*}
	&\alpha\left(\tau^{x_{3}}\alpha_{1}\right)^{b_{4}^{-1}}\alpha^{-1}=\rho^{\kappa_{1}}\tau^{x_{3}}\alpha_{1} \ \text{and} \\
	&\alpha\left(\tau^{x_{3}}\alpha_{1}\right)^{ab_{1}b_{3}b_{4}^{-2}-r_{3}b_{4}^{-1}-\frac{1}{2}b_{4}^{-1}\left(1-b_{1}\right)-\frac{1}{2}ab_{1}b_{4}^{-1}\left(b_{1}b_{4}^{-1}-1\right)}\left(y\alpha_{2}\alpha_{3}^{a}\right)^{b_{1}b_{4}^{-1}}\alpha^{-1}\\
	&=\rho^{\kappa_{2}}\sigma^{b_{1}^{2}b_{4}^{-1}}\tau^{\left(ab_{1}b_{3}b_{4}^{-2}-r_{3}b_{4}^{-1}-\frac{1}{2}b_{4}^{-1}\left(1-b_{1}\right)-\frac{1}{2}ab_{1}b_{4}^{-1}\left(b_{1}b_{4}^{-1}-1\right)\right)x_{3}+\frac{1}{2}b_{1}\left(b_{1}b_{4}^{-1}-1\right)}\alpha_{2}\alpha_{3}^{ab_{1}^{2}b_{4}^{-1}},
	\end{align*}
	we must have $ b_{4}=b_{1}^{2} $ and 
	\[r_{3}=ab_{1}^{-1}b_{3}+\frac{1}{2}\left(b_{1}-1\right)\left(1+a\right)+\frac{1}{2}b_{1}^{2}x_{3}^{-1}\left(b_{1}+1\right);\]
	thus we must have
	\[\mathrm{Aut}_{\mathcal{B}r}(\left\langle\rho,\tau^{x_{3}}\alpha_{1},\sigma\alpha_{2}\alpha_{3}^{a} \right\rangle)=\left\{ \alpha \in \mathrm{Aut}(M_{1}) \mid \alpha= \alpha_{1}^{r_{1}}\alpha_{3}^{ab_{1}^{-1}b_{3}+\frac{1}{2}\left(b_{1}-1\right)\left(1+a\right)+\frac{1}{2}b_{1}^{2}x_{3}^{-1}\left(b_{1}+1\right)} \begin{psmallmatrix} b_{1} & 0 \\ b_{3} & b_{1}^{2} \end{psmallmatrix} \right\}.\]
	
	Therefore, we have 
	\begin{align*}
	&e(M_{1},M_{1},p^{2})= \sum_{(M_{1})_{M_{1}}(p^{2})}\dfrac{\left\lvert\mathrm{Aut}(M_{1})\right\rvert}{\left\lvert\mathrm{Aut}_{\mathcal{B}r}((M_{1})_{M_{1}})\right\rvert}=\\
	&\sum_{u_{2}\neq 0, 4}\dfrac{\left\lvert\mathrm{Aut}(M_{1})\right\rvert}{\left\lvert\mathrm{Aut}_{\mathcal{B}r}(\left\langle \rho, \sigma\alpha_{1}, \sigma^{u_{2}}\tau^{\frac{u_{2}^{2}}{4}}\alpha_{3} \right\rangle)\right\rvert}+\sum_{\substack{u_{2}-u_{3}, u_{3}, u_{2}^{2}-4u_{3}\neq 0 \\u_{2}^{2}-4u_{3} \ \text{is a square}} }\dfrac{\left\lvert\mathrm{Aut}(M_{1})\right\rvert}{ \left\lvert\mathrm{Aut}_{\mathcal{B}r}(\left\langle \rho, \sigma\alpha_{1}, \sigma^{u_{2}}\tau^{u_{3}}\alpha_{3} \right\rangle)\right\rvert}+\\
	&\sum_{\substack{u_{2}-u_{3}, u_{3}, u_{2}^{2}-4u_{3}\neq 0 \\u_{2}^{2}-4u_{3} \ \text{is not a square}} }\dfrac{\left\lvert\mathrm{Aut}(M_{1})\right\rvert}{ \left\lvert\mathrm{Aut}_{\mathcal{B}r}(\left\langle \rho, \sigma\alpha_{1}, \sigma^{u_{2}}\tau^{u_{2}}\alpha_{3} \right\rangle)\right\rvert}+\sum_{v_{2}\neq 0, 2}\dfrac{\left\lvert\mathrm{Aut}(M_{1})\right\rvert}{\left\lvert\mathrm{Aut}_{\mathcal{B}r}(\left\langle \rho, \tau^{-v_{2}}\alpha_{1}, \sigma^{v_{2}}\alpha_{3} \right\rangle)\right\rvert}\\
	&+\sum_{\substack{x_{3}\neq 0, \ a\\ (1+x_{3})x_{3}^{-1}\neq a}}\dfrac{\left\lvert\mathrm{Aut}(M_{1})\right\rvert}{ \left\lvert\mathrm{Aut}_{\mathcal{B}r}(\left\langle\rho,\tau^{x_{3}}\alpha_{1},\sigma\alpha_{2}\alpha_{3}^{a} \right\rangle)\right\rvert}\\
	&=(p^{2}-1)(p^{2}-p)p^{2}\times\\
	&\left(\dfrac{p-2}{(p-1)p^{3}}+\dfrac{\frac{p-1}{2}+\left(\frac{p-1}{2}-1\right)(p-2)}{(p-1)^{2}p^{2}}+\dfrac{\frac{p-1}{2}+\left(\frac{p-1}{2}\right)(p-2)}{(p^{2}-1)p^{2}}+ \dfrac{p-2}{(p^{2}-1)(p^{2}-p)p^{2}}+\dfrac{(p-1)^{2}}{(p-1)p^{2}}\right)\\
	&=(p^{4}-p^{3}-2p^{2}+2p+1)p,
	\end{align*}
	and similarly 
	\begin{align*}
	&e(C_{p}^{3},M_{1},p^{2})= \sum_{(C_{p}^{3})_{M_{1}}(p^{2})}\dfrac{\left\lvert\mathrm{Aut}(C_{p}^{3})\right\rvert}{\left\lvert\mathrm{Aut}_{\mathcal{B}r}((C_{p}^{3})_{M_{1}})\right\rvert}=\\
	&\dfrac{\left\lvert\mathrm{Aut}(C_{p}^{3})\right\rvert}{\left\lvert\mathrm{Aut}_{\mathcal{B}r}(\left\langle \rho, \sigma\alpha_{1}, \sigma^{4}\tau^{4}\alpha_{3} \right\rangle)\right\rvert}+\sum_{\substack{ u_{2}^{2}-4u_{2}\neq 0 \\ \text{is a square}} }\dfrac{\left\lvert\mathrm{Aut}(C_{p}^{3})\right\rvert}{ \left\lvert\mathrm{Aut}_{\mathcal{B}r}(\left\langle \rho, \sigma\alpha_{1}, \sigma^{u_{2}}\tau^{u_{2}}\alpha_{3} \right\rangle)\right\rvert}+\\
	&\sum_{\substack{u_{2}^{2}-4u_{2}\neq 0 \\ \text{is not a square}} }\dfrac{\left\lvert\mathrm{Aut}(C_{p}^{3})\right\rvert}{ \left\lvert\mathrm{Aut}_{\mathcal{B}r}(\left\langle \rho, \sigma\alpha_{1}, \sigma^{u_{2}}\tau^{u_{2}}\alpha_{3} \right\rangle)\right\rvert}+\dfrac{\left\lvert\mathrm{Aut}(C_{p}^{3})\right\rvert}{\left\lvert\mathrm{Aut}_{\mathcal{B}r}(\left\langle \rho, \tau^{-2}\alpha_{1}, \sigma^{2}\alpha_{3} \right\rangle\right\rvert}\\
	&+\sum_{x_{3}\neq 0}\dfrac{\left\lvert\mathrm{Aut}(C_{p}^{3})\right\rvert}{ \left\lvert\mathrm{Aut}_{\mathcal{B}r}(\left\langle\rho,\tau^{x_{3}}\alpha_{1},\sigma\alpha_{2}\alpha_{3}^{(1+x_{3})x_{3}^{-1}} \right\rangle)\right\rvert}\end{align*}
	\begin{align*}
	&=(p^{3}-1)(p^{3}-p)(p^{3}-p^{2})\times\\
	&\left(\dfrac{1}{(p-1)p^{3}}+\dfrac{\frac{p-1}{2}-1}{(p-1)^{2}p^{2}}+\dfrac{\frac{p-1}{2}}{(p^{2}-1)p^{2}}+ \dfrac{1}{(p^{2}-1)(p^{2}-p)p^{2}}+\dfrac{p-1}{(p-1)p^{2}}\right)\\
	&=(p^{3}-1)(p^{2}-2)p^{2}.
	\end{align*}
\end{proof}
\begin{lemma}\label{L18}
	For $ \left\lvert\varTheta(G) \right\rvert=p^{3} $ there are exactly four  $ M_{1} $-skew braces of $ M_{1} $ type and no other. Furthermore, there are only \[ (p^{2}-1)p^{3} \] Hopf-Galois structures of $ M_{1} $ type on Galois extensions of fields with Galois group $ G \cong M_{1} $ and $ \left\lvert\varTheta(G) \right\rvert=p^{3} $. 
\end{lemma}
\begin{proof}
	If $ G \subseteq \mathrm{Hol}(M_{1}) $ with $ \left\lvert\varTheta(G) \right\rvert=p^{3} $, then we can assume, without loss of generality, that $ \varTheta(G) = \left\langle\alpha_{1},\alpha_{2},\alpha_{3}\right\rangle $, and so \[G=\left\langle u\alpha_{1},v\alpha_{2},w\alpha_{3}\right\rangle \]
	where $ u=\rho^{u_{1}}\sigma^{u_{2}}\tau^{u_{3}} $, $ v=\rho^{v_{1}}\sigma^{v_{2}}\tau^{v_{3}} $, $ w=\rho^{w_{1}}\sigma^{w_{2}}\tau^{w_{3}} $, and $ G $ is isomorphic to $ \varTheta(G)\cong M_{1} $. Now
	\begin{align*}
	\left(u\alpha_{1}\right)\left(v\alpha_{2}\right)&=\rho^{v_{2}}uv\alpha_{1}\alpha_{2}\ \text{and} \\
	\left(v\alpha_{2}\right)\left(u\alpha_{1}\right)&=\rho^{\frac{1}{2}u_{2}\left(u_{2}-1\right)+v_{3}u_{2}-u_{3}v_{2}-u_{2}^{2}-u_{2}v_{2}}\tau^{u_{2}}uv\alpha_{1}\alpha_{2}, 
	\end{align*}
	so we need $ u_{2}=0 $ and $ v_{2}\equiv-u_{3}v_{2}\ \mathrm{mod} \ p $. We have 
	\begin{align*}
	\left(u\alpha_{1}\right)\left(w\alpha_{3}\right)&=\rho^{w_{2}}uw\alpha_{1}\alpha_{3}\ \text{and} \\
	\left(w\alpha_{3}\right)\left(u\alpha_{1}\right)&=\rho^{u_{3}+w_{3}u_{2}-u_{3}w_{2}}uw\alpha_{1}\alpha_{3},
	\end{align*}
	so, since $ u_{2}=0 $, we need $ w_{2}\equiv u_{3}-u_{3}w_{2} \ \mathrm{mod} \ p $. Finally, we have 
	\begin{align*}
	\left(u\alpha_{1}\right)\left(v\alpha_{2}\right)\left(w\alpha_{3}\right)&=\left(\rho^{v_{2}}uv\alpha_{1}\alpha_{2}\right)\left(w\alpha_{3}\right)\\
	&=\rho^{u_{1}-v_{2}(w_{2}-1)-\frac{1}{2}w_{2}(w_{2}-1)}\tau^{u_{3}+w_{2}}vw\alpha_{1}\alpha_{2}\alpha_{3} \ \text{and}\\
	\left(w\alpha_{3}\right)\left(v\alpha_{2}\right)&=\rho^{v_{3}+w_{3}v_{2}-v_{3}w_{2}}vw\alpha_{3}\alpha_{2}, 
	\end{align*}
	so we need $ u_{3}+w_{2}\equiv 0 \ \mathrm{mod} \ p $ and 
	\[u_{1}-v_{2}(w_{2}-1)-\frac{1}{2}w_{2}(w_{2}-1) \equiv v_{3}+w_{3}v_{2}-v_{3}w_{2} \ \mathrm{mod} \ p.\]
	Combining the above information, for $ G $ to be a group of order $ p^{3} $, we need, modulo $ p $,
	\begin{align}\label{E63}
	&u_{2}=0, \ v_{2}=-u_{3}v_{2}, \ w_{2}=u_{3}-u_{3}w_{2}, \ u_{3}=-w_{2}, \nonumber\\
	&u_{1}-v_{2}(w_{2}-1)-\frac{1}{2}w_{2}(w_{2}-1) = v_{3}+w_{3}v_{2}-v_{3}w_{2}.
	\end{align}
	Now the equations $ w_{2}=u_{3}-u_{3}w_{2} $ and $ u_{3}=-w_{2} $ imply that \[ u_{3}=-w_{2}=0,-2. \] Given this, the equation $ v_{2}=-u_{3}v_{2} $ implies that $ v_{2}=0 $. Now the final equation in (\ref{E63}) reduces to 
	\[u_{1}-\frac{1}{2}w_{2}(w_{2}-1) = v_{3}-v_{3}w_{2}.\]
	
	Thus, we can consider two cases for $ w_{2}=0 $ and $ w_{2}=2 $. If $ w_{2}=0 $, then $ u,v $ and $ w $ are of the form
	\[u=\rho^{u_{1}}, \ v=\rho^{v_{1}}\tau^{u_{1}}, \ w=\rho^{w_{1}}\tau^{w_{3}}, \]
	and in this case $ G $ cannot be regular. Therefore, we must set $ w_{2}=2 $, hence $ u,v $, and $ w $ are of the form
	\[u=\rho^{u_{1}}\tau^{-2}, \ v=\rho^{v_{1}}\tau^{1-u_{1}}, \ w=\rho^{w_{1}}\sigma^{2}\tau^{w_{3}}. \]
	Now for $ G $ to be regular we need 
	\[\left(u\alpha_{1}\right)^{\frac{1}{2}(1-u_{1})}\left(w\alpha_{3}\right)=\rho^{v_{1}+\frac{1}{2}u_{1}(1-u_{1})}\alpha_{1}^{\frac{1}{2}(1-u_{1})}\alpha_{3}\not\in \mathrm{Aut}(M_{1}),\]
	so we need $ v_{1}+\frac{1}{2}u_{1}(1-u_{1})\not\equiv 0  \ \mathrm{mod} \ p $. Therefore, $ G $ is conjugate to 
	\[\left\langle \rho^{u_{1}}\tau^{-2}\alpha_{1},\rho^{v_{1}}\tau^{1-u_{1}}\alpha_{2},\rho^{w_{1}}\sigma^{2}\tau^{w_{3}}\alpha_{3}\right\rangle\cong M_{1} \]\[ \text{for} \ u_{1},v_{1},w_{1},w_{3}=0,...,p-1 \ \text{with} \ v_{1}+\frac{1}{2}u_{1}(1-u_{1})\not\equiv 0  \ \mathrm{mod} \ p,\]
	and there are (taking into account the $ p+1 $ conjugates) 
	\[(p+1)(p-1)p^{3}\]
	of these. 
	
	To find the non-isomorphic skew braces corresponding to the above regular subgroups, it suffices to conjugate by automorphisms of the form $ \alpha\stackrel{\mathrm{def}}{=}\beta\gamma \in \mathrm{Aut}(M_{1}) $, where $ \beta \stackrel{\mathrm{def}}{=} \begin{psmallmatrix} b_{1} & 0 \\ b_{3} & b_{4} \end{psmallmatrix}  \in \mathrm{GL}_{2}(\mathbb{F}_{p}) $ and $ \gamma\stackrel{\mathrm{def}}{=} \alpha_{1}^{r_{1}}\alpha_{3}^{r_{3}} \in C_{p}^{2}  $. Now using (\ref{E112}) and (\ref{E16}) we have 
	\begin{align*}
	\alpha\left(u\alpha_{1}\right)^{b_{4}^{-1}}\alpha^{-1}&=\left(\alpha\cdot u^{b_{4}^{-1}}\right)\alpha_{1},\\
	\alpha\left(v\alpha_{2}\right)^{b_{1}b_{4}^{-1}}\alpha^{-1}&=\left(\alpha\cdot v^{b_{1}b_{4}^{-1}}\right)\alpha_{1}^{r_{3}+\frac{1}{2}\left(1-b_{1}\right)}\alpha_{2},\\
	\alpha\left(w\alpha_{3}\right)^{b_{1}^{-1}}\alpha^{-1}&=\left(\alpha\cdot\left(\rho^{\frac{1}{2}w_{3}b_{1}^{-1}\left(b_{1}^{-1}-1\right)}w^{b_{1}^{-1}}\right)\right)\alpha_{1}^{-b_{1}^{-1}b_{3}}\alpha_{3}, 
	\end{align*}
	so we have 
	\begin{align*}
	&\alpha\left(u\alpha_{1}\right)^{b_{4}^{-1}}\alpha^{-1}=\left(\alpha\cdot u^{b_{4}^{-1}}\right)\alpha_{1},\\
	&\alpha\left(u\alpha_{1}\right)^{-r_{3}b_{4}^{-1}-\frac{1}{2}b_{4}^{-1}(1-b_{1})}\left(v\alpha_{2}\right)^{b_{1}b_{4}^{-1}}\alpha^{-1}=\left(\alpha\cdot \left(u^{-r_{3}b_{4}^{-1}-\frac{1}{2}b_{4}^{-1}(1-b_{1})}v^{b_{1}b_{4}^{-1}}\right)\right)\alpha_{2},\\
	&\alpha\left(u\alpha_{1}\right)^{b_{1}^{-1}b_{3}b_{4}^{-1}}\left(w\alpha_{3}\right)^{b_{1}^{-1}}\alpha^{-1}=\left(\left(\alpha\cdot u^{b_{1}^{-1}b_{3}b_{4}^{-1}}\right)\alpha\alpha_{1}^{b_{1}^{-1}b_{3}b_{4}^{-1}}\cdot\left(\rho^{\frac{1}{2}w_{3}b_{1}^{-1}\left(b_{1}^{-1}-1\right)}w^{b_{1}^{-1}}\right)\right)\alpha_{3}. 
	\end{align*}
	Note that we have \[ \alpha=\begin{bsmallmatrix} b_{1}b_{4} & \frac{1}{2}b_{1}b_{3}+r_{1}b_{1}+r_{3}b_{3} & r_{3}b_{4} \\ 0 & b_{1} & 0 \\ 0 & b_{3} & b_{4} \end{bsmallmatrix},\] 
	We let $ b_{5}\stackrel{\mathrm{def}}{=}\frac{1}{2}b_{1}b_{3}+r_{1}b_{1}+r_{3}b_{3} $. Now  
	\begin{align*}
	&\alpha\cdot u^{b_{4}^{-1}}=\rho^{u_{1}b_{1}-2r_{3}}\tau^{-2},\\
	\alpha\cdot &\left(u^{-r_{3}b_{4}^{-1}-\frac{1}{2}b_{4}^{-1}(1-b_{1})}v^{b_{1}b_{4}^{-1}}\right)=\rho^{r_{3}\left(2r_{3}+1\right)+v_{1}b_{1}^{2}+\frac{1}{2}u_{1}b_{1}(b_{1}-1)-2r_{3}u_{1}b_{1}}\tau^{1+2r_{3}-u_{1}b_{1}},\\
	&\left(\alpha\cdot u^{b_{1}^{-1}b_{3}b_{4}^{-1}}\right)\left(\alpha\alpha_{1}^{b_{1}^{-1}b_{3}b_{4}^{-1}}\cdot\left(\rho^{\frac{1}{2}w_{3}b_{1}^{-1}\left(b_{1}^{-1}-1\right)}w^{b_{1}^{-1}}\right)\right)=\rho^{b_{3}u_{1}-2r_{3}b_{1}^{-1}b_{3}}\tau^{-2b_{1}^{-1}b_{3}}\\
	&\rho^{\frac{3}{2}w_{3}b_{4}\left(b_{1}^{-1}-1\right)+b_{4}w_{1}+2b_{1}^{-1}b_{3}+2b_{1}^{-1}b_{5}+b_{3}(2b_{1}^{-1}-1)}\sigma^{2}\tau^{2b_{1}^{-1}b_{3}+w_{3}b_{1}^{-1}b_{4}}\\
	&=\rho^{2r_{1}+\frac{3}{2}w_{3}b_{4}\left(b_{1}^{-1}-1\right)+b_{4}w_{1}+u_{1}b_{3}}\sigma^{2}\tau^{w_{3}b_{1}^{-1}b_{4}}.
	\end{align*}
	We let 
	\begin{align*}
	r_{1}&=-\frac{3}{4}w_{3}b_{4}\left(b_{1}^{-1}-1\right)-\frac{1}{2}b_{4}w_{1}-\frac{1}{2}u_{1}b_{3},\\
	r_{3}&=\frac{1}{2}u_{1}b_{1},
	\end{align*}
	which gives us 
	\begin{align*}
	&\alpha\cdot u^{b_{4}^{-1}}=\tau^{-2},\\
	\alpha\cdot &\left(u^{-r_{3}b_{4}^{-1}-\frac{1}{2}b_{4}^{-1}(1-b_{1})}v^{b_{1}b_{4}^{-1}}\right)=\rho^{\left(v_{1}+\frac{1}{2}u_{1}\left(1-u_{1}\right)\right)b_{1}^{2}}\tau,\\
	&\left(\alpha\cdot u^{b_{1}^{-1}b_{3}b_{4}^{-1}}\right)\left(\alpha\alpha_{1}^{b_{1}^{-1}b_{3}b_{4}^{-1}}\cdot\left(\rho^{\frac{1}{2}w_{3}b_{1}^{-1}\left(b_{1}^{-1}-1\right)}w^{b_{1}^{-1}}\right)\right)=\sigma^{2}\tau^{w_{3}b_{1}^{-1}b_{4}}.
	\end{align*}
	Next, for a fixed $ \delta \in \mathbb{F}_{p}^{\times} $ which is not a square, we can write \[ \left(v_{1}+\frac{1}{2}u_{1}\left(1-u_{1}\right)\right)=s_{1}^{2}s \] where $ s_{1}\in \mathbb{F}_{p}^{\times}  $ and $ s=1,\delta $. Letting $ b_{1}=\pm s_{1}^{-1} $ we get 
	\begin{align*}
	&\alpha\cdot u^{b_{4}^{-1}}=\tau^{-2},\\
	&\alpha\cdot \left(u^{-r_{3}b_{4}^{-1}-\frac{1}{2}b_{4}^{-1}(1-b_{1})}v^{b_{1}b_{4}^{-1}}\right)=\rho^{s}\tau,\\
	&\left(\alpha\cdot u^{b_{1}^{-1}b_{3}b_{4}^{-1}}\right)\left(\alpha\alpha_{1}^{b_{1}^{-1}b_{3}b_{4}^{-1}}\cdot\left(\rho^{\frac{1}{2}w_{3}b_{1}^{-1}\left(b_{1}^{-1}-1\right)}w^{b_{1}^{-1}}\right)\right)=\sigma^{2}\tau^{\pm s_{1}w_{3}b_{4}}.
	\end{align*}
	
	Therefore, every such regular subgroup is conjugate to 
	\begin{align}\label{E3}
	\left\langle\tau^{-2}\alpha_{1},\rho^{s}\tau\alpha_{2},\sigma^{2}\tau^{t_{3}}\alpha_{3}\right\rangle\cong M_{1} \ \text{for} \ t_{3}=0,1, \ s=1,\delta,
	\end{align}
	and these subgroups are not further conjugate to each other, so they give us four non-isomorphic skew braces.
	
	To find the number of corresponding Hopf-Galois structures, we need to find the automorphism groups of above skew braces. We let \[ \alpha=\gamma\beta \in \mathrm{Aut}(M_{1})\ \text{where}\  \gamma\stackrel{\mathrm{def}}{=} \alpha_{1}^{r_{1}}\alpha_{3}^{r_{3}}, \ \beta \stackrel{\mathrm{def}}{=} \begin{pmatrix} b_{1} & b_{2} \\ b_{3} & b_{4} \end{pmatrix},\]
	and set $ b_{2}=0 $. If $ \alpha \in \mathrm{Aut}_{\mathcal{B}r}(\left\langle\tau^{-2}\alpha_{1},\rho^{s}\tau\alpha_{2},\sigma^{2}\tau^{t_{3}}\alpha_{3}\right\rangle) $, since by our notation above we have
	\begin{align*}
	&\alpha\cdot u^{b_{4}^{-1}}=\rho^{-2r_{3}}\tau^{-2},\\
	\alpha\cdot &\left(u^{-r_{3}b_{4}^{-1}-\frac{1}{2}b_{4}^{-1}(1-b_{1})}v^{b_{1}b_{4}^{-1}}\right)=\rho^{r_{3}\left(2r_{3}+1\right)+sb_{1}^{2}}\tau^{1+2r_{3}},\\
	&\left(\alpha\cdot u^{b_{1}^{-1}b_{3}b_{4}^{-1}}\right)\left(\alpha\alpha_{1}^{b_{1}^{-1}b_{3}b_{4}^{-1}}\cdot\left(\rho^{\frac{1}{2}t_{3}b_{1}^{-1}\left(b_{1}^{-1}-1\right)}w^{b_{1}^{-1}}\right)\right)=\rho^{2r_{1}+\frac{3}{2}t_{3}b_{4}\left(b_{1}^{-1}-1\right)}\sigma^{2}\tau^{t_{3}b_{1}^{-1}b_{4}}.
	\end{align*}
	we must have $ r_{3}=0 $, $ b_{1}^{2}=1 $, $ r_{1}=\frac{3}{4}t_{3}b_{4}\left(1-b_{1}^{-1}\right) $, further $ b_{1}=b_{4} $ if $ t_{3}=1 $. Therefore, we have
	\begin{align*}
	\mathrm{Aut}_{\mathcal{B}r}(\left\langle\tau^{-2}\alpha_{1},\rho^{s}\tau\alpha_{2},\sigma^{2}\alpha_{3}\right\rangle)&=\left\{ \alpha \in \mathrm{Aut}(M_{1}) \mid \alpha= \begin{psmallmatrix} \pm 1 & 0 \\ b_{3} & b_{4} \end{psmallmatrix} \right\},\\
	\mathrm{Aut}_{\mathcal{B}r}(\left\langle\tau^{-2}\alpha_{1},\rho^{s}\tau\alpha_{2},\tau\sigma^{2}\alpha_{3}\right\rangle)&=\left\{ \alpha \in \mathrm{Aut}(M_{1}) \mid \alpha= \alpha_{1}^{\frac{3}{4}\left(\pm 1-1\right)}\begin{psmallmatrix} \pm 1 & 0 \\ b_{3} & \pm 1 \end{psmallmatrix} \right\}.
	\end{align*} 
	Now again we find 
	\begin{align*}
	 &e(M_{1},M_{1},p^{3})= \sum_{(M_{1})_{M_{1}}(p^{3})}\dfrac{\left\lvert\mathrm{Aut}(M_{1})\right\rvert}{\left\lvert\mathrm{Aut}_{\mathcal{B}r}((M_{1})_{M_{1}}(p^{3})\right\rvert}=\\
	 &\dfrac{2\left\lvert\mathrm{Aut}(M_{1})\right\rvert}{\left\lvert\mathrm{Aut}_{\mathcal{B}r}(\left\langle\tau^{-2}\alpha_{1},\rho\tau\alpha_{2},\sigma^{2}\alpha_{3}\right\rangle)\right\rvert}+\dfrac{2\left\lvert\mathrm{Aut}(M_{1})\right\rvert}{ \left\lvert\mathrm{Aut}_{\mathcal{B}r}(\left\langle\tau^{-2}\alpha_{1},\rho\tau\alpha_{2},\tau\sigma^{2}\alpha_{3}\right\rangle)\right\rvert}\\
	 &=\dfrac{2(p^{2}-1)(p-1)p^{3}}{2(p-1)p}+\dfrac{2(p^{2}-1)(p-1)p^{3}}{2p}=(p^{2}-1)p^{3}.
	 \end{align*} 
\end{proof}
\subsection{Socle and annihilator of skew braces of $ M_{1} $ type}
	Finally, we note that from our classification of skew braces we are also able to determine their \textit{socle} and \textit{annihilator}. Let 
	$ B=\left(B,\oplus,\odot\right) $ be a skew brace. As before we let 
	\begin{align*}
	m: \left(B,\odot\right) &\longrightarrow \mathrm{Hol}\left(B,\oplus\right)\\
	a&\longmapsto \left(m_{a} : b \longmapsto a\odot b\right)
	\end{align*}
	and set
	\begin{align*}
	\varTheta : \mathrm{Hol}\left(B,\oplus\right)& \longrightarrow \mathrm{Aut}\left(B,\oplus\right)\\
	\eta\alpha&\longmapsto \alpha.
	\end{align*} 
	We shall denote by $ \lambda = \varTheta m $. Then $ \Ker \lambda = \Ima m \cap \left(B,\oplus\right) $ inside $ \mathrm{Hol}\left(B,\oplus\right) $.
	 
	 First we note that \cite[cf.][p.~ 23]{MR3763907} an \textit{ideal} of a skew brace $ B=\left(B,\oplus,\odot\right) $ is defined to be a subset $ I \subseteq B $, such that $ I $ is a normal subgroup with respect to both operations $ \oplus $ and $ \odot $, and $ \lambda_{a}(I)\subseteq I $ for all $ a \in B $.  The \textit{socle} of $ B $ is defined to be
	\[\mathrm{Soc}(B)\stackrel{\mathrm{def}}{=}\{a\in B\mid a\oplus b=a\odot b, \ b\oplus(b\odot a)=(b\odot a)\oplus b\ \text{for all} \ b \in B \},\]
	which is an ideal of $ B $, and one has $ \mathrm{Soc}(B)=\Ker \lambda \cap \mathrm{Z} \left(B,\oplus\right) $. Finally, \cite[cf.][Definition 7]{doi:10.1142/S0219498819500336}, the \textit{annihilator} of $ B $ is defined to be 
	\[\mathrm{Ann}(B)\stackrel{\mathrm{def}}{=}\mathrm{Soc}(B)\cap \mathrm{Z} \left(B,\odot\right)=\Ker \lambda \cap \mathrm{Z} \left(B,\oplus\right)\cap\mathrm{Z} \left(B,\odot\right),\]
	which is also an ideal of $ B $.
	
	Now we aim to explain what each of these terms, ideal, socle, and annihilator, correspond to if we are given a regular subgroup $ H \subseteq \mathrm{Hol}\left(N\right) $ and we consider it as a skew brace. Recall first from Subsection \ref{SB2}, given a regular subgroup $ H \subseteq \mathrm{Hol}\left(N\right) $, it can be represented as
	\[H=\left\langle \eta_{1},...,\eta_{r},v_{1}\alpha_{1},...,v_{s}\alpha_{s} \right\rangle,\]
	for $ H_{1}\stackrel{\mathrm{def}}{=}\left\langle \eta_{1},...,\eta_{r}\right\rangle\subseteq N $ and $ H_{2}\stackrel{\mathrm{def}}{=}\left\langle \alpha_{1},...,\alpha_{s} \right\rangle \subseteq \mathrm{Aut}\left(N\right) $ and some $ v_{1},...,v_{s} \in N $. Note also that we have a bijection 
	\begin{align*}
	\psi:H&\longrightarrow N \\
	g&\longmapsto g_{1}\stackrel{\mathrm{def}}{=}g(1_{N}).
	\end{align*} 
	To get a skew brace we can set $ \left(H,\odot\right)=H $ and define $ \oplus $ on $ H $ by 
	\[g\oplus h=\psi^{-1}\left(g_{1}h_{1}\right), \] 
	which makes $ \left(H,\oplus,\odot\right) $ into a skew brace with $ \left(H,\oplus\right)\stackrel{\psi}{\cong}N $. Note the map $ \psi $ now induces an isomorphism 
	\begin{align*}
	\mathrm{Hol}\left(H,\oplus\right)&\longrightarrow \mathrm{Hol}\left(N\right) \\
	g\beta&\longmapsto g_{1}\psi\beta\psi^{-1},
	\end{align*} 
	which maps $ \Ker \lambda $ to $ H_{1} $, and $ \Ima \lambda $ to $ H_{2} $.
	
	Now for a subset $ I \subseteq H $ to be an ideal of $ H $ considered as a skew brace, we need $ I \subseteq \left(H,\odot\right)  $ to be a normal subgroup, $ \psi\left(I\right) \subseteq N $ to be a normal subgroup (so $ I \subseteq \psi^{-1}\left(N\right)=\left(H,\oplus\right) $ is a normal subgroup) and $ H_{2}\left(\psi\left(I\right)\right)\subseteq \psi\left(I\right)  $. Furthermore, one has \[ \mathrm{Soc}(H)= \Ker \lambda \cap \mathrm{Z} \left(H,\oplus\right)=H_{1} \cap \psi^{-1}\left(Z\left(N\right)\right), \]
	and 
	\[\mathrm{Ann}(H)=H_{1} \cap \psi^{-1}\left(Z\left(N\right)\right)\cap\mathrm{Z} \left(H\right).\]
	
	Recall the skew braces of $ M_{1} $ type, apart from the trivial skew brace $ \langle\rho,\sigma,\tau\rangle $, as found in Lemmas \ref{L16}, \ref{L17}, \ref{L18} are as follows. 
	\begin{itemize}
		\item For $ \left\lvert\Ker \lambda \right\rvert=p^{2} $ form Lemma \ref{L16}, (\ref{E54}) we have non-isomorphic skew braces
		\begin{align*}
		&\left\langle \rho,\tau, \sigma\alpha_{3} \right\rangle, \left\langle\rho, \tau, \sigma\alpha_{2}\alpha_{3}  \right\rangle   \cong C_{p}^{3} , \  \left\langle  \rho, \tau, \sigma\alpha_{1}  \right\rangle, \left\langle \rho,\tau, \sigma\alpha_{2} \right\rangle,\nonumber\\
		&\left\langle \rho,\tau, \sigma\alpha_{3}^{c} \right\rangle, \left\langle\rho, \tau, \sigma\alpha_{2}\alpha_{3}^{c}  \right\rangle  \cong M_{1} \ \text{for} \ c=2,...,p-1, 
		\end{align*}
		so in all these cases we have 
		\[\mathrm{Soc}(H)=\mathrm{Ann}(H)=\left\langle \rho \right\rangle.\] 
		\item For $ \left\lvert\Ker \lambda \right\rvert=p $ from Lemma \ref{L17}, (\ref{E57}) and (\ref{E62}), we have non-isomorphic skew braces \[\left\langle \rho, \sigma\alpha_{1}, \sigma^{u_{3}}\tau^{u_{4}}\alpha_{3} \right\rangle,\left\langle \rho, \tau^{-u_{5}}\alpha_{1}, \sigma^{u_{5}}\alpha_{3} \right\rangle,\left\langle \rho, \tau^{x_{3}}\alpha_{1}, \sigma\alpha_{2}\alpha_{3}^{a} \right\rangle \cong M_{1}, \]
		\[\left\langle \rho, \sigma\alpha_{1}, \sigma^{u_{2}}\tau^{u_{2}}\alpha_{3} \right\rangle,\left\langle \rho, \tau^{-2}\alpha_{1}, \sigma^{2}\alpha_{3} \right\rangle,\left\langle \rho, \tau^{x_{3}}\alpha_{1}, \sigma\alpha_{2}\alpha_{3}^{\left(1+x_{3}\right)x_{3}^{-1}} \right\rangle \cong C_{p}^{3} \ \text{for} \]
		\[a,u_{3}=0,...,p-1, \  u_{2},u_{4},u_{5},x_{3},=1,...,p-1 \]
		\[\text{with}  \ u_{5}\neq2, \ u_{3}-u_{4},\ ax_{3}-\left(1+x_{3}\right)\not\equiv 0 \ \mathrm{mod}\ p,  \]
		so in all these cases we also have 
		\[\mathrm{Soc}(H)=\mathrm{Ann}(H)=\left\langle \rho \right\rangle.\] 
		\item For $ \left\lvert\Ker \lambda \right\rvert=1 $ from Lemma \ref{L18}, (\ref{E3}) we have non-isomorphic skew braces \[\left\langle\tau^{-2}\alpha_{1},\rho^{s}\tau\alpha_{2},\sigma^{2}\tau^{t_{3}}\alpha_{3}\right\rangle\cong M_{1} \ \text{for} \ t_{3}=0,1, \ s=1,\delta,\]
		so in all these cases have
		\[\mathrm{Soc}(H)=\mathrm{Ann}(H)=1.\] 
	\end{itemize}	
\begin{center}
	\large\textbf{Acknowledgements}
\end{center} The author is ever indebted to Prof Nigel Byott and Prof Agata  Smoktunowicz for their continued support and useful suggestions. The author is ever grateful for the referee's comments which lead to numerous improvements to the manuscript.

This research was partially supported by the ERC Advanced grant 320974. The author obtained part of the results in this paper while studying for a PhD degree at the University of Exeter funded by an EPSRC Doctoral Training Grant.
\section{References}
\bibliographystyle{elsarticle-num} 
\bibliography{mybibfile}

\begin{thebibliography}{10}
\expandafter\ifx\csname url\endcsname\relax
  \def\url#1{\texttt{#1}}\fi
\expandafter\ifx\csname urlprefix\endcsname\relax\def\urlprefix{URL }\fi
\expandafter\ifx\csname href\endcsname\relax
  \def\href#1#2{#2} \def\path#1{#1}\fi

\bibitem{MR2278047}
W.~Rump, \href{http://dx.doi.org/10.1016/j.jalgebra.2006.03.040}{Braces,
  radical rings, and the quantum {Y}ang-{B}axter equation}, J. Algebra 307~(1)
  (2007) 153--170.
\newblock \href {http://dx.doi.org/10.1016/j.jalgebra.2006.03.040}
  {\path{doi:10.1016/j.jalgebra.2006.03.040}}.
\newline\urlprefix\url{http://dx.doi.org/10.1016/j.jalgebra.2006.03.040}

\bibitem{MR3177933}
F.~Ced\'o, E.~Jespers, J.~Okni\'nski,
  \href{http://dx.doi.org/10.1007/s00220-014-1935-y}{Braces and the
  {Y}ang-{B}axter equation}, Comm. Math. Phys. 327~(1) (2014) 101--116.
\newblock \href {http://dx.doi.org/10.1007/s00220-014-1935-y}
  {\path{doi:10.1007/s00220-014-1935-y}}.
\newline\urlprefix\url{http://dx.doi.org/10.1007/s00220-014-1935-y}

\bibitem{MR3527540}
D.~Bachiller, F.~Ced\'o, E.~Jespers,
  \href{http://dx.doi.org/10.1016/j.jalgebra.2016.05.024}{Solutions of the
  {Y}ang-{B}axter equation associated with a left brace}, J. Algebra 463 (2016)
  80--102.
\newblock \href {http://dx.doi.org/10.1016/j.jalgebra.2016.05.024}
  {\path{doi:10.1016/j.jalgebra.2016.05.024}}.
\newline\urlprefix\url{http://dx.doi.org/10.1016/j.jalgebra.2016.05.024}

\bibitem{MR3647970}
L.~Guarnieri, L.~Vendramin, \href{http://dx.doi.org/10.1090/mcom/3161}{Skew
  braces and the {Y}ang-{B}axter equation}, Math. Comp. 86~(307) (2017)
  2519--2534.
\newblock \href {http://dx.doi.org/10.1090/mcom/3161}
  {\path{doi:10.1090/mcom/3161}}.
\newline\urlprefix\url{http://dx.doi.org/10.1090/mcom/3161}

\bibitem{MR3763907}
A.~Smoktunowicz, L.~Vendramin, \href{https://doi.org/10.4171/JCA/2-1-3}{On skew
  braces (with an appendix by {N}. {B}yott and {L}. {V}endramin)}, J. Comb.
  Algebra 2~(1) (2018) 47--86.
\newblock \href {http://dx.doi.org/10.4171/JCA/2-1-3}
  {\path{doi:10.4171/JCA/2-1-3}}.
\newline\urlprefix\url{https://doi.org/10.4171/JCA/2-1-3}

\bibitem{MR0260724}
S.~U. Chase, M.~E. Sweedler, Hopf algebras and {G}alois theory, Lecture Notes
  in Mathematics, Vol. 97, Springer-Verlag, Berlin-New York, 1969.

\bibitem{MR878476}
C.~Greither, B.~Pareigis,
  \href{http://dx.doi.org/10.1016/0021-8693(87)90029-9}{Hopf-{G}alois theory
  for separable field extensions}, J. Algebra 106~(1) (1987) 239--258.
\newblock \href {http://dx.doi.org/10.1016/0021-8693(87)90029-9}
  {\path{doi:10.1016/0021-8693(87)90029-9}}.
\newline\urlprefix\url{http://dx.doi.org/10.1016/0021-8693(87)90029-9}

\bibitem{MR1402555}
N.~P. Byott, \href{http://dx.doi.org/10.1080/00927879608825743}{Uniqueness of
  {H}opf-{G}alois structure for separable field extensions}, Comm. Algebra
  24~(10) (1996) 3217--3228.
\newblock \href {http://dx.doi.org/10.1080/00927879608825743}
  {\path{doi:10.1080/00927879608825743}}.
\newline\urlprefix\url{http://dx.doi.org/10.1080/00927879608825743}

\bibitem{MR2030805}
N.~P. Byott, \href{http://dx.doi.org/10.1016/j.jpaa.2003.10.010}{Hopf-{G}alois
  structures on {G}alois field extensions of degree {$pq$}}, J. Pure Appl.
  Algebra 188~(1-3) (2004) 45--57.
\newblock \href {http://dx.doi.org/10.1016/j.jpaa.2003.10.010}
  {\path{doi:10.1016/j.jpaa.2003.10.010}}.
\newline\urlprefix\url{http://dx.doi.org/10.1016/j.jpaa.2003.10.010}

\bibitem{MR2363137}
N.~P. Byott,
  \href{http://dx.doi.org/10.1016/j.jalgebra.2007.04.010}{Hopf-{G}alois
  structures on almost cyclic field extensions of 2-power degree}, J. Algebra
  318~(1) (2007) 351--371.
\newblock \href {http://dx.doi.org/10.1016/j.jalgebra.2007.04.010}
  {\path{doi:10.1016/j.jalgebra.2007.04.010}}.
\newline\urlprefix\url{http://dx.doi.org/10.1016/j.jalgebra.2007.04.010}

\bibitem{MR3715201}
A.~A. Alabdali, N.~P. Byott,
  \href{https://doi.org/10.1016/j.jalgebra.2017.09.009}{Counting
  {H}opf-{G}alois structures on cyclic field extensions of squarefree degree},
  J. Algebra 493 (2018) 1--19.
\newblock \href {http://dx.doi.org/10.1016/j.jalgebra.2017.09.009}
  {\path{doi:10.1016/j.jalgebra.2017.09.009}}.
\newline\urlprefix\url{https://doi.org/10.1016/j.jalgebra.2017.09.009}

\bibitem{MR1704676}
S.~Carnahan, L.~Childs,
  \href{http://dx.doi.org/10.1006/jabr.1999.7861}{Counting {H}opf-{G}alois
  structures on non-abelian {G}alois field extensions}, J. Algebra 218~(1)
  (1999) 81--92.
\newblock \href {http://dx.doi.org/10.1006/jabr.1999.7861}
  {\path{doi:10.1006/jabr.1999.7861}}.
\newline\urlprefix\url{http://dx.doi.org/10.1006/jabr.1999.7861}

\bibitem{MR1644203}
T.~Kohl, \href{http://dx.doi.org/10.1006/jabr.1998.7479}{Classification of the
  {H}opf-{G}alois structures on prime power radical extensions}, J. Algebra
  207~(2) (1998) 525--546.
\newblock \href {http://dx.doi.org/10.1006/jabr.1998.7479}
  {\path{doi:10.1006/jabr.1998.7479}}.
\newline\urlprefix\url{http://dx.doi.org/10.1006/jabr.1998.7479}

\bibitem{CS}
T.~Crespo, M.~Salguero, \href{https://arxiv.org/abs/1807.11409}{{H}opf {G}alois
  structures on separable field extensions of odd prime power degree}, Preprint
  on ArXiv.org.
\newline\urlprefix\url{https://arxiv.org/abs/1807.11409}

\bibitem{MR2298848}
W.~Rump, \href{http://dx.doi.org/10.1016/j.jpaa.2006.07.001}{Classification of
  cyclic braces}, J. Pure Appl. Algebra 209~(3) (2007) 671--685.
\newblock \href {http://dx.doi.org/10.1016/j.jpaa.2006.07.001}
  {\path{doi:10.1016/j.jpaa.2006.07.001}}.
\newline\urlprefix\url{http://dx.doi.org/10.1016/j.jpaa.2006.07.001}

\bibitem{MR3320237}
D.~Bachiller,
  \href{http://dx.doi.org/10.1016/j.jpaa.2014.12.013}{Classification of braces
  of order {$p^3$}}, J. Pure Appl. Algebra 219~(8) (2015) 3568--3603.
\newblock \href {http://dx.doi.org/10.1016/j.jpaa.2014.12.013}
  {\path{doi:10.1016/j.jpaa.2014.12.013}}.
\newline\urlprefix\url{http://dx.doi.org/10.1016/j.jpaa.2014.12.013}

\bibitem{doi:10.1142/S0219498819500336}
F.~Catino, I.~Colazzo, P.~Stefanelli,
  \href{https://www.worldscientific.com/doi/abs/10.1142/S0219498819500336}{Skew
  left braces with non-trivial annihilator}, Journal of Algebra and Its
  Applications\href {http://dx.doi.org/10.1142/S0219498819500336}
  {\path{doi:10.1142/S0219498819500336}}.
\newline\urlprefix\url{https://www.worldscientific.com/doi/abs/10.1142/S0219498819500336}

\bibitem{CD}
C.~Dietzel, \href{https://arxiv.org/abs/1801.06911}{Braces of order $ p^{2}q
  $}, Preprint on ArXiv.org.
\newline\urlprefix\url{https://arxiv.org/abs/1801.06911}

\bibitem{KSV}
A.~Konovalov, A.~Smoktunowicz, V.~L.,
  \href{https://arxiv.org/abs/1804.04106}{On skew braces and their ideals},
  Preprint on ArXiv.org.
\newline\urlprefix\url{https://arxiv.org/abs/1804.04106}

\bibitem{SS}
A.~Smoktunowicz~(Agata), A.~Smoktunowicz~(Alicia),
  \href{https://arxiv.org/abs/1704.03558}{Set-theoretic solutions of the
  {Y}ang-{B}axter equation and new classes of {R}-matrices}, Preprint on
  ArXiv.org.
\newline\urlprefix\url{https://arxiv.org/abs/1704.03558}

\bibitem{LV}
L.~Vendramin, \href{https://arxiv.org/abs/1807.06411}{{P}roblems on skew left
  braces}, Preprint on ArXiv.org.
\newline\urlprefix\url{https://arxiv.org/abs/1807.06411}

\bibitem{KNZ}
K.~Nejabati~Zenouz,
  \href{https://ore.exeter.ac.uk/repository/handle/10871/32248}{On Hopf-Galois
  Structures and Skew Braces of Order $ p^{3} $}, The University of Exeter, PhD
  Thesis, Supervised by Prof N. Byott, Funded by EPSRC DTG, January 2018.
\newline\urlprefix\url{https://ore.exeter.ac.uk/repository/handle/10871/32248}

\bibitem{MR1879021}
N.~P. Byott, \href{https://doi.org/10.1006/jabr.2001.9053}{Integral
  {H}opf-{G}alois structures on degree {$p^2$} extensions of {$p$}-adic
  fields}, J. Algebra 248~(1) (2002) 334--365.
\newblock \href {http://dx.doi.org/10.1006/jabr.2001.9053}
  {\path{doi:10.1006/jabr.2001.9053}}.
\newline\urlprefix\url{https://doi.org/10.1006/jabr.2001.9053}

\bibitem{MR1767499}
L.~N. Childs, \href{http://dx.doi.org/10.1090/surv/080}{Taming wild extensions:
  {H}opf algebras and local {G}alois module theory}, Vol.~80 of Mathematical
  Surveys and Monographs, American Mathematical Society, Providence, RI, 2000.
\newblock \href {http://dx.doi.org/10.1090/surv/080}
  {\path{doi:10.1090/surv/080}}.
\newline\urlprefix\url{http://dx.doi.org/10.1090/surv/080}

\bibitem{DVO}
D.~V. Osipov, \href{https://arxiv.org/abs/1505.00348}{{D}iscrete {H}eisenberg
  group and its automorphism group}, Preprint on ArXiv.org.
\newline\urlprefix\url{https://arxiv.org/abs/1505.00348}

\end{thebibliography}
\end{document}